\documentclass{amsart}
\usepackage{graphicx}

\newtheorem{thm}{Theorem}[section]
\newtheorem{exm}[thm]{Example}

\newtheorem{lem}[thm]{Lemma}

\theoremstyle{definition}
\newtheorem{defn}[thm]{Definition}
\theoremstyle{remark}
\newtheorem{rem}[thm]{Remark}
\numberwithin{equation}{section}
\newcommand{\Real}{\mathbb R}

\begin{document}

\title{The fractional d'Alembert's formulas}
\author{Cheng-Gang Li}
\address{Department of Mathematics, Southwest Jiaotong  University, Chengdu 611756, P.R.China.}
\email{lichenggang@swjtu.edu.cn}
\author{Miao Li}
\address{Department of Mathematics, Sichuan University}
\email{mli@scu.edu.cn}
\author{Sergey Piskarev}
\address{Science Research Computer Center, Lomonosov Moscow State University}
\email{piskarev@gmail.com}
\author{Mark {M. Meerschaert} }
\address{Department of Statistics and probability, Michigan State University}
\email{mcubed@stt.msu.edu}
\thanks{The first author is supported by the Fundamental Research Funds for the Central Universities
(No. A0920502051820-62), the second author is supported by the NSFC-RFBR Programme of China (No. 11611530677), and
the third author was partially supported  by grants of Russian Foundation for  Basic Research  $15-01-00026\_a, 16-01-00039\_a$,
 $17-51-53008{\_}a$. {The fourth author was partially supported by US ARO grant W911NF-15-1-0562 and US NSF grants EAR-1344280 and DMS-1462156.}}
\date{}

\maketitle
\begin{abstract}
In this paper we develop generalized d'Alembert's formulas for abstract fractional integro-differential equations and fractional differential equations
on Banach spaces.
Some examples are given to illustrate our abstract results, and the probability interpretation of these fractional d'Alembert's formulas are also given. Moreover, we also provide
d'Alembert's formulas for abstract fractional telegraph equations.
\end{abstract}

\section{Introduction}

It is well-known that the solution of traditional wave equation on the line
\begin{equation}\label{wave-equation}
\begin{cases}
u_{tt}(t,x) = u_{xx}(t,x), \quad t>0, \,x \in \Real\\
u(0,x)= \phi(x), \quad u_t(0,x)= \psi(x)
\end{cases}
\end{equation}
is given by d'Alembert's formula
\begin{equation}\label{dAlmbert}
u(t,x) = \frac{1}{2}[\phi(x+t) + \phi(x-t)] + \frac{1}{2}\int_{x-t}^{x+t} \psi(y)dy.
\end{equation}

Including a forcing function, the solution of the wave equation on the line
\begin{equation}\label{wave1}
\begin{cases}
w_{tt}(t,x) = w_{xx}(t,x) + f(t,x) , \quad t>0, \,x \in \Real \\
w(0,x)= 0, \quad w_t(0,x)= 0
\end{cases}
\end{equation}
is given by the Duhamel's principle formula
$$ w(t,x) = \int_0^t r(t,x,\tau) d \tau , $$
where $ r(t,x,\tau )$ is the solution of wave equation
\begin{equation}\label{wave2}
\begin{cases}
r_{tt}(t,x,\tau) = r_{xx}(t,x,\tau), \quad t>0, \,x \in \Real \\
r(\tau,x,\tau)= 0, \quad r_t(\tau,x, \tau)= f(\tau,x)
\end{cases}
\end{equation}
The fractional   Duhamel's principle formula was obtained by \cite{forcing,Umarov}.

It is also of interest to know a fractional version of the d'Alembert formula. Next we rewrite (\ref{wave-equation}) as an integral equation,
which is more {easily} generalized to the fractional case. By integrating the wave equation (\ref{wave-equation}) twice with respect to $t$, we get the following integro-differential equation
$$
u(t,x)= \phi(x) + t\psi(x) + \int_0^t(t-s)u_{xx}(s,x)ds.
$$
One possible generalization of the above equation to fractional order
$1\le \alpha \le 2$ is
\begin{equation}\label{Fujitaequation}
u(t,x)= \phi(x) + \frac{t^{\alpha/2}}{\Gamma(1+\frac{\alpha}{2})}\psi(x) + \frac{1}{\Gamma(\alpha)}\int_0^t(t-s)^{\alpha-1}u_{xx}(s,x)ds.
\end{equation}
Fujita studied the above equation in \cite{Fujita} and showed that the unique solution is given by
\begin{equation}\label{Fujitaformula}
u(t,x) = \frac{1}{2}{\bf E}[\phi(x+ Y_{\alpha/2}(t)) + \phi(x-Y_{\alpha/2}(t))] + \frac{1}{2}{\bf E} \int_{x-Y_{\alpha/2}(t)}^{x+Y_{\alpha/2}(t)} \psi(y)dy,
\end{equation}
where $Y_{\alpha/2}(t)=\sup_{0\leq s\leq t}X_\alpha (s)$, and $X_\alpha (t)$ $(1\leq\alpha\leq2)$ is a $c\grave{a}dl\grave{a}g$ stable process with
characteristic function ${\bf E}\exp\{i\xi X_\alpha (t)\}= \exp\{-t|\xi|^{2/\alpha}e^{-(\pi i/2)(2-2/\alpha){\rm sgn}(\xi)}\}$.
$Y_{\alpha/2}(t)$ can also be regarded as the inverse of an $\alpha/2$ stable subordinator \cite{meerschaert, meerschaert2013}.
When $\alpha=2$, the expression (\ref{Fujitaformula}) reduces to d'Alembert's formula (\ref{dAlmbert}). Fujita also mentioned that ${\bf E}[\phi(x \pm Y_{\alpha/2}(t))\pm \int_{0}^{x\pm Y_{\alpha/2}(t)} \psi(y)dy]$
%and ${\bf E}[\phi(x- Y_{\alpha/2}(t))- \int_{0}^{x-Y_{\alpha/2}(t)} \psi(y)dy]$ are
are solutions for
\begin{equation}\label{Fujitaequation1}
u^\pm(t,x)= \phi(x) \pm \int_0^x \psi(y)dy \pm \frac{1}{\Gamma(\alpha/2)}\int_0^t(t-s)^{\alpha/2-1}u^{\pm}_{x}(s,x)ds,
\end{equation}
respectively, and the solution of (\ref{Fujitaequation}) can be decomposed as
$
u(t,x) = \frac{1}{2}(u^+(t,x) + u^-(t,x)).
$

Next we convert the integro-differential equation (\ref{Fujitaequation}) to a fractional differential equation.  See Section 2 for the definition of fractional derivatives, fractional integrals, and the special functions $g_\alpha(t)$. We refer to \cite{Kilbas, meerschaert2012book, Podlubny, Samko}
for more details on fractional derivatives and fractional differential equations. Now if $u$ satisfies the equation (\ref{Fujitaequation}), then by differentiating it with respect to $t$ for $\alpha/2$-times in the sense of Caputo fractional derivatives and  by using the identity $D_t^\alpha 1=0$, we have
\begin{eqnarray*}
D_t^{\alpha/2} u(t,x) &=& D_t^{\alpha/2}(u(t,x)- \phi(x))\\
&=& D_t^{\alpha/2}(g_{1+{\alpha/2}}(t)\psi(x) + (J_t^\alpha u_{xx})(t,x))\\
&=& \psi(x) + (J_t^{\alpha/2}u_{xx})(t,x),
\end{eqnarray*}
and next differentiating for $\alpha/2$-times again we get
\begin{eqnarray*}
D_t^{\alpha/2} (D_t^{\alpha/2} u(t,x)) &=& D_t^{\alpha/2} (\psi(x)+(J_t^{\alpha/2}u_{xx})(t,x))\\
&=&D_t^{\alpha/2}(J_t^{\alpha/2}u_{xx})(t,x)) = u_{xx}(t,x).
\end{eqnarray*}
This suggests an $\alpha$-order differential equation $D_t^{\alpha/2} D_t^{\alpha/2} u = u_{xx}$.
It is also interesting to consider the integro-differential equation $D_t^{\alpha} u = u_{xx}$, because
$D_t^{\alpha} = D_t^{\alpha/2}D_t^{\alpha/2}$ holds only under some special conditions.
%the problem is that it may happen that $D_t^{\alpha/2}D_t^{\alpha/2}u\not= D_t^{\alpha}u$ if $\alpha \not=2$.

Motivated by the above observations, we will first consider d'Alembert's formula for abstract fractional integro-differential equation in the form of
\begin{equation} \label{0alpha}
u(t)= \phi + \frac{t^{\alpha/2}}{\Gamma(1+\alpha/2)}\psi + \frac{1}{\Gamma(\alpha)}\int_0^t(t-s)^{\alpha-1}A^2 u(s)ds, \quad t >0
\end{equation}
on a Banach space. It is known that the well-posedness of equation (\ref{0alpha}) is equivalent to the existence of an $\alpha$-times resolvent family for $A^2$.
Thus the theory of fractional resolvent families will be our main tool. The notion of resolvent families was first introduced by Pr\"uss \cite{Pruss} to study Volterra integral equations, and then developed systematically by Bajlekova \cite{Baj} for fractional Cauchy problems. The fractional resolvent families can be considered as generalizations of $C_0$-semigroups and cosine operator functions \cite{ABHN, Engel1999, Pazy}.
The d'Alembert formula for wave equation is in fact the decomposition of a cosine operator function, see for example \cite[Chapter III]{Krein}. For its fractional analogue
 we will use the decomposition theorem for fractional resolvent families \cite{LZ},
%The most important fact we will use is the decomposition theorem for fractional resolvent families \cite{LZ},
i.e. our Lemma \ref{decom}. Thanks to this lemma, we are able to give the solution of (\ref{0alpha}) in Theorem \ref{abFujita} and decompose the solution as $u = \frac{1}{2}(u^++ u^-)$, where $u^\pm$ are solutions to
\begin{equation}\label{Fujitaequation2}
u^\pm(t,x)= \phi(x) \pm A^{-1}\psi \pm \frac{1}{\Gamma(\alpha/2)}\int_0^t(t-s)^{\alpha/2-1}Au^{\pm}(s,x)ds.
\end{equation}
respectively, when $\psi$ is in the range of the operator $A$. When $A = \frac{\partial}{\partial x}$, then (\ref{Fujitaequation2}) is the same as (\ref{Fujitaequation1}). The corresponding fractional differential equations for (\ref{0alpha}) and (\ref{Fujitaequation2}) are
\begin{equation*}\label{0alphaf}
\begin{cases}
D_t^{\alpha/2} D_t^{\alpha/2} u(t) = A^2 u(t), \quad t>0\\
u(0)= \phi, \quad D_t^{\alpha/2}u(0)= \psi
\end{cases}
\end{equation*}
and
\begin{equation*}\label{0alphaf}
\begin{cases}
D_t^{\alpha/2} u^\pm(t) = \pm A u^\pm (t), \quad t>0\\
u^\pm(0)= \phi \pm A^{-1} \psi
\end{cases}
\end{equation*}
respectively.

In Theorems \ref{thmab} and \ref{afie} we will construct the d'Alembert formula for the more general equation
\begin{equation*} \label{0general betaA}
u(t)= \phi + \frac{t^{\beta}}{\Gamma(1+\beta)}\psi + \frac{1}{\Gamma(\alpha)}\int_0^t(t-s)^{\alpha-1}A^2 u(s)ds, \quad t >0
\end{equation*}
or its corresponding fractional differential equation
\begin{equation*}\label{0two-equation2}
\begin{cases}
D_t^{\alpha- \beta} D_t^{\beta} u(t) = A^2 u(t), \quad t>0\\
u(0)= \phi, \quad D_t^{\beta}u(0)= \psi.
\end{cases}
\end{equation*}
In particular, when $\beta= \alpha/2$, an alternative  d'Alembert formula for (\ref{0alpha}) will be provided. More precisely, the solution of (\ref{0alpha})
%It is remarkable that the operator $D_t^{\alpha/2} D_t^{\alpha/2}-A^2$ can be decomposed into the product of two operators $D_t^{\alpha/2}-A$ and $D_t^{\alpha/2}+A$, therefore we are
 can be decomposed as $u(t) = \frac{1}{2}(v^+(t) + v^-(t))$, where $v^\pm$ are solutions to
\begin{equation*}
v^\pm(t)= \phi + \frac{t^{\alpha/2}}{\Gamma(1+\alpha/2)}\psi \pm \frac{1}{\Gamma(\alpha/2)}\int_0^t(t-s)^{\alpha/2-1}A v^\pm(s)ds, \quad t >0,
\end{equation*}
and the corresponding fractional differential equations are
\begin{equation*}
\begin{cases}
D_t^{\alpha/2}  v^\pm(t) = \pm A v^\pm(t) + \psi, \quad t>0\\
v^\pm(0)= \phi.
\end{cases}
\end{equation*}
And the d'Alembert formula for fractional differential equation like
\begin{equation*}
\begin{cases}
D_t^{\alpha} u(t) = A^2 u(t), \quad t>0\\
u(0)= \phi, \quad u_t(0)= \psi
\end{cases}
\end{equation*}
will also be considered.

Our papers is organized as follows: in Section 2 we will recall some results on fractional resolvent families and then derive the fractional d'Alembert's formula for
abstract fractional integro-differential equations and fractional differential equations on Banach spaces; some concrete examples are given in Section 3 to illustrate our abstract results,
and their probability interpretations are also given; finally in Section 4 we will give the fractional  d'Alembert's formula for fractional telegraph equations.

\section{d'Alembert's formula for abstract fractional equations}

Let $X$ be a Banach space and $A$ be a closed linear densely defined operator on $X$. We begin with the definitions of fractional integrals and derivatives.

\begin{defn}
The Riemann-Liouville fractional integral of order $\alpha > 0$ is defined as
$$
J^\alpha_t f(t):= (g_\alpha\ast f)(t) = \int_0^t g_\alpha(t-s)f(s)ds,\quad\quad f\in L^1([0,+\infty);X),\quad t>0
$$
where
\begin{equation*}
g_\alpha(t)=\left\{
\begin{array}{ll}
\frac{ t^{\alpha-1}}{\Gamma(\alpha)},&t>0, \\
0,&t\leq 0.
\end{array}
\right.
\end{equation*}
Set moreover $J^0_t f(t)=f(t)$.
\end{defn}

%Next we turn to consider abstract fractional Cauchy problems.

\begin{defn}
The Riemann-Liouville fractional derivative of order $\alpha > 0$ is defined as
$$
{}_{RL}D^\alpha_t f(t):= D^m_tJ^{m-\alpha}_tf(t),\quad t \in (0,T)
$$
for $m-1<\alpha\leq m$, $m$ is an integer, $f\in L^1((0,T);X)$, and $g_{m-\alpha}\ast f\in W^{m,1}((0,T);X)$, where
\begin{eqnarray*}
W^{m,1}((0,T);X)&:=&\{ f| \mbox{ there exists } \phi \in L^1((0,T);X) \mbox{ such that }
\\
&&\quad\quad
f(t) = \sum_{k=0}^{m-1}c_k g_{k+1}(t)+ (g_m*\phi)(t),\, t \in (0,T)\}.
\end{eqnarray*}
The Caputo fractional derivative of order $\alpha > 0$ is defined as
$$
D^\alpha_t f(t):= J^{m-\alpha}_tD^m_tf(t), \quad t \in (0,T)
$$
if $f\in W^{m,1}((0,T);X)$. Moreover, we define
$$
{}_{RL}D^\alpha_t f(0):=\lim_{t \to 0} {}_{RL}D^\alpha_t f(t), \quad D^\alpha_t f(0):=\lim_{t \to 0} D^\alpha_t f(t)
$$
if the limits exist.
\end{defn}

%The relation between Riemann-Liouville fractional derivative and Caputo fractional derivative reads
%\begin{equation}
%D^\alpha_t f(t)= {}_{RL}D^\alpha_t \big(f(t)-\sum_{k=0}^{m-1}f^{(k)} (0) g_{k+1}(t) \big).
%\end{equation}

%Let $0< \alpha \le 2$.

Now we consider the Volterra equation
\begin{equation}\label{Volterra}
u(t) = f(t) + \int_0^t g_\alpha(t-s) Au(s)ds, \quad t \ge 0
\end{equation}
where $f(t)$ is a continuous $X$-valued function.

\begin{defn} Let $u(t) : \Real_+ \to X$ be continuous.

(1) $u(t)$ is called a strong solution of (\ref{Volterra}) if $u(t) \in D(A)$  and (\ref{Volterra}) holds for $t \ge 0$;

(2) $u(t)$ is called a mild solution of (\ref{Volterra}) if $(g_{\alpha} *u)(t) \in D(A)$ and $u(t) = f(t) + A(g_{\alpha} *u)(t) $ for $t \ge 0$.
\end{defn}

The solution family for (\ref{Volterra}) is defined by \cite{Baj, Pruss}.
\begin{defn}
A family $\{S_\alpha(t)\}_{t \ge 0}\subset B(X)$ is called an $\alpha$-times resolvent family for the operator $A$ (or generated by $A$) if the following conditions are satisfied:

$(1)$ $S_\alpha(t)x: \Real_+ \to X$ is continuous for every $x \in X$ and $S_\alpha(0)=I$;

$(2)$ $S_\alpha (t) D(A) \subset D(A)$ and $AS_\alpha(t) x = S_\alpha(t)Ax$ for all $x \in D(A)$ and $t \ge 0$;

$(3)$ the resolvent equation
\begin{equation}\label{A-resolventequation}
S_\alpha(t)x = x + (g_\alpha*S_\alpha)(t)Ax
\end{equation}
holds for every $x \in D(A)$.
\end{defn}

\begin{rem}\label{remfirst}
Since $A$ is closed and densely defined, it is easy to show that for all $x \in X$, $(g_\alpha*S_\alpha)(t)x \in D(A)$ and
\begin{equation}\label{X-resolventequation}
S_\alpha(t)x = x + A(g_\alpha*S_\alpha)(t)x.
\end{equation}
\end{rem}

It is shown in \cite{Pruss} that the Volterra equation (\ref{Volterra}) is well-posed if and only if the operator $A$ generates an $\alpha$-times resolvent family $S_\alpha(t)$, and the mild solution to (\ref{Volterra}) is
given by
\begin{equation}\label{Prussf}
u(t) = \frac{d}{dt} \int_0^t S_\alpha(t-s) f(s) ds.
\end{equation}
In particular, the mild solution to
\begin{equation}\label{ugalphaA}
u(t)= x + \int_0^t g_\alpha(t-s) Au(s)ds
\end{equation}
is given by $u(t)= S_\alpha(t)x$; in addition, if $x \in D(A)$, then $u(t)$ is also a strong solution. By differentiating (\ref{ugalphaA}) $\alpha$-times, we get a fractional equation of $\alpha$-order
\begin{equation}\label{alphaequation}
\begin{cases}
D_t^{\alpha}u(t) = Au(t), \quad t>0\\
u(0)= x\quad (u_t(0)= 0 \mbox{ if } 1<\alpha \le 2).
\end{cases}
\end{equation}
It is also known that the well-posedness of (\ref{alphaequation}) is equivalent to that of (\ref{Volterra}), and thus is equivalent to the existence of an $\alpha$-times resolvent family $S_\alpha(t)$ for $A$. In this case,
the unique mild solution of (\ref{alphaequation}) is also given by $u(t)= S_\alpha(t)x$,  For  details we refer to \cite{Baj}.\\

Now we recall the following result on the generation of fractional resolvent families, which is crucial for our decomposition theorem.

\begin{lem} \cite{LZ}\label{decom}
Let $0< \alpha \le 2$. Suppose that both $A$ and $-A$ generate $\alpha/2$-times resolvent families $S_{\alpha/2}^+(t)$ and $S_{\alpha/2}^-(t)$, respectively. Then $A^2$ generates an $\alpha$-times resolvent family
$S_\alpha(t)$, which is given by  $S_\alpha(t) = \frac{1}{2}[S_{\alpha/2}^+(t)+ S_{\alpha/2}^-(t)]$.
\end{lem}

\begin{rem}
(1) If $A$ generates a $C_0$-group then, by the subordination principle for fractional resolvent families \cite[Theorem 3.1]{Baj}, both $A$ and $-A$ generate $\alpha/2$-times resolvent families. Thus the generator of a
$C_0$-group satisfies the assumptions in the above lemma.

(2) Let $1< \alpha <2$. Suppose that there is some $\theta >0$ with $0 < \theta < \min\{\frac{\pi}{2}, \frac{\pi}{\alpha}- \frac{\pi}{2}\}$ such that
$$
\sigma (A) \subset \{z \in {\mathbb C}: \frac{\alpha}{2}(\frac{\pi}{2} + \theta)\le |\arg z| \le \pi - \frac{\alpha}{2}(\frac{\pi}{2} + \theta)\}=:\Gamma_{\alpha,\theta}
$$
and for every $\theta' > \theta$, there is a constant $M_{\theta'}$ such that
$$
\|z(z-A)^{-1}\| \le M_{\theta'}, \quad z \in {\mathbb C} - \Gamma_{\alpha,\theta'}.
$$
This is equivalent to saying that  both $A$ and $-A$ are sectorial operators with angle $\pi - \frac{\alpha}{2}(\frac{\pi}{2} + \theta)$. Then by \cite[Lemma 2.7]{LCL}, both $A$ and $-A$ generate bounded analytic
$\alpha/2$-times resolvent families of angle $\theta$, and $A^2$ also generates a bounded analytic $\alpha$-times resolvent family of angle $\theta$. The converse is also true by \cite[Proposition 5.6]{CLL}. For the case
that $\alpha= 2$, we recall the fact that the generator of an analytic cosine function is always a bounded operator. It is interesting here to mention a result of Fattorini \cite{Fattorini}: on a UMD space $X$, if $A^2$
generates a bounded cosine function, then $A$ generates a $C_0$-group. It is not clear whether a similar result holds for the generator of a fractional resolvent family.
\end{rem}

%With some additional hypotheses, the converse of Lemma \ref{decom} also holds.
%\begin{lem}\cite{CLL}\label{decom}
%Let $0< \alpha \le 2$. If $A^2$ generates a bounded analytic $\alpha$-times resolvent family $S_\alpha(t)$, then both $A$ and $-A$ generate $\alpha/2$-times resolvent families $S_{\alpha/2}^+(t)$ and $S_{\alpha/2}^-(t)$,
%respectively, and $S_\alpha(t) = \frac{1}{2}[S_{\alpha/2}^+(t)+ S_{\alpha/2}^-(t)]$.
%\end{lem}

Let us begin with the d'Alembert formula for an abstract version of (\ref{Fujitaequation}).
\begin{thm}\label{abFujita}
Let $1 < \alpha \le 2$. Suppose that both $A$ and $-A$ generate $\alpha/2$-times resolvent families $S_{\alpha/2}^+(t)$ and $S_{\alpha/2}^-(t)$ on $X$, respectively. Suppose also that $\phi \in X$, $ \psi \in R(A)$ and $\psi = A\Psi$ for some $\Psi \in D(A)$. Then
 the unique mild solution of the following integro-differential equation
\begin{equation}\label{Fujitaabstract}
u(t)= \phi + \frac{t^{\alpha/2}}{\Gamma(1+\frac{\alpha}{2})}\psi + \frac{1}{\Gamma(\alpha)}\int_0^t(t-s)^{\alpha-1}A^2u(s)ds.
\end{equation}
is given by
\begin{equation}\label{Fdalembert}
u(t) = \frac{1}{2}[S_{\alpha/2}^+(t) \phi + S_{\alpha/2}^- (t)\phi]+\frac{1}{2}[S_{\alpha/2}^+(t) \Psi - S_{\alpha/2}^-(t)\Psi].
\end{equation}
And the solution $u$ can be decomposed as $u= \frac{1}{2}(u^+ + u^-)$, where $u^+$ and $u^-$ are mild solutions to
\begin{equation}\label{abFujitahalf1}
u^+(t)= \phi + \Psi + \frac{1}{\Gamma(\alpha/2)}\int_0^t(t-s)^{\alpha/2-1}A u^+(s)ds, \quad t >0,
\end{equation}
and
\begin{equation}\label{abFujitahalf2}
u^-(t)= \phi - \Psi - \frac{1}{\Gamma(\alpha/2)}\int_0^t(t-s)^{\alpha/2-1}A u^-(s)ds, \quad t >0,
\end{equation}
respectively. Moreover, the corresponding fractional differential equation for (\ref{abFujita}) is
\begin{equation}
\begin{cases}
D_t^{\alpha/2} D_t^{\alpha/2} u(t) =  A^2 u(t), \quad t>0\\
u(0)=\phi, \quad D_t^{\alpha/2} u(0) = \psi,
\end{cases}
\end{equation}
and the fractional differential equation corresponding to (\ref{abFujitahalf1}) and (\ref{abFujitahalf2}) are
\begin{equation}\label{forduhamelfirst}
\begin{cases}
D_t^{\alpha/2}  u^+(t) =  A u^+(t), \quad t>0\\
u^+(0)= \phi+ \Psi.
\end{cases}
\end{equation}
and
\begin{equation}\label{forduhamelsecond}
\begin{cases}
D_t^{\alpha/2}  u^-(t) =  A u^-(t), \quad t>0\\
u^-(0)= \phi- \Psi.
\end{cases}
\end{equation}
respectively.
\end{thm}
\begin{proof}
 Suppose that $u^+$ and $u^-$ are mild solutions of (\ref{abFujitahalf1}) and (\ref{abFujitahalf2}) respectively. Then $J_t^{\alpha/2}u^+, J_t^{\alpha/2}u^- \in D(A)$ and  the integral equation
$$
u^+  = \phi +  \Psi+ AJ^{\alpha/2}_t u^+
$$
and
$$
u^-  = \phi -  \Psi- AJ^{\alpha/2}_t u^-
$$
hold. Thus,
\begin{eqnarray*}
 u^+ + u^- &=& 2\phi + A J_t^{\alpha/2}(u^+ - u^-) \\
%&=& \phi + A  J_t^{\alpha/2} [\frac{1}{2}(\phi+ \Psi) + A J_t^{\alpha/2}u^+ - \frac{1}{2}(\phi- \Psi) + A J_t^{\alpha/2}u^- ]\\
&=& 2\phi +  A  J_t^{\alpha/2} [2 \Psi + A J_t^{\alpha/2}(u^+ + u^-)];
\end{eqnarray*}
since $ \Psi \in D(A)$, this implies that $J_t^\alpha(u^+ + u^-) \in D(A^2)$ and
$$
u^+ + u^- = 2\phi +2 J_t^{\alpha/2} \psi + A^2 J_t^\alpha(u^+ + u^-).
$$
Therefore, the function $u := \frac{1}{2}(u^+ + u^-)$ is the mild solution for (\ref{Fujitaabstract}). Moreover, since both $A$ and $-A$ generate $\alpha/2$-times resolvent families,
$$
u^+(t) = S_{\alpha/2}^+(t) (\phi + \Psi), \quad u^-(t) = S_{\alpha/2}^-(t) (\phi - \Psi),
$$
this gives the representation (\ref{Fdalembert}). The uniqueness of the mild solution follows from the well-posedness
of the integro-differential equation
$$
u(t)= y +  \frac{1}{\Gamma(\alpha)}\int_0^t(t-s)^{\alpha-1}A^2u(s)ds.
$$
where $y \in X$, which is guaranteed by Lemma \ref{decom}. The corresponding fractional differential equations can be derived by differentiation for fractional times.
\end{proof}

\begin{rem}\label{rem2.6}
(1) When $\alpha=2$ and $A= d/dx$ on the line, then
$$
S_{1}^+(t) \phi(x)= \phi(x+t), \quad S_{1}^-(t) \phi(x)= \phi(x-t),
$$
$\Psi(x) = \int_0^x \psi(y)dy$, and
$$
S_{\alpha/2}^+(t) \Psi(x) - S_{\alpha/2}^-(t)\Psi(x)= \int_0^{x+t} \psi(y)dy - \int_0^{x-t} \psi(y) dy = \int_{x-t}^{x+t} \psi(y)dy.
$$
Hence the formula (\ref{Fdalembert}) is exactly the classical d'Alembert formula (\ref{dAlmbert}). Thus (\ref{Fdalembert}) can be considered as the d'Alembert formula for the abstract integro-differential equation (\ref{Fujitaabstract}).

(2) When $A = d/dx$ on the line, the equations (\ref{abFujitahalf1}) and (\ref{abFujitahalf2}) are the same as Fujita's equation (\ref{Fujitaequation1}), so our decomposition can be viewed as the abstract version of Fujita's
decomposition.

(3) When $\alpha =2$, our expression (\ref{Fdalembert}) reduces to the formula (1.13) in \cite[Chapter 3]{Krein} given by Krein.

 %One can also derive (\ref{Fdalembert}) for the Volterra equation (\ref{Fujitaabstract}) by the method of Pr\"uss. Let $f(t)=\phi + \frac{t^{\alpha/2}}{\Gamma(1+{\alpha\over 2})}\psi$, then we have $f: [0,+\infty)\mapsto X$, and %$f\in W^{1,1}([0,T], X)$ for every $T > 0$,
%by Proposition 1.2 on Page 33-34 in \cite{Pruss}, the integral equation (\ref{Fujitaabstract}) has a mild (strong) solution
%\begin{eqnarray*}
%u(t)&= &S_\alpha (t)f(0)+ \int_0^tS_\alpha(t-s)f^\prime (s)ds\\
%&=&S_\alpha (t)\phi+\int_0^tS_\alpha(t-s)\frac{s^{\alpha/2-1}}{\Gamma(\alpha/2)}\psi ds \\
%&=& \frac{1}{2}(S_{\alpha/2}^+(t) + S_{\alpha/2}^-(t))\phi + \frac{1}{2}g_{\alpha/2}* (S_{\alpha/2}^+(t) + S_{\alpha/2}^-(t)) A\Psi\\
%&=& \frac{1}{2}(S_{\alpha/2}^+(t) + S_{\alpha/2}^-(t))\phi + \frac{1}{2}Ag_{\alpha/2}* (S_{\alpha/2}^+(t) + S_{\alpha/2}^-(t)) \Psi\\
%&=& \frac{1}{2}(S_{\alpha/2}^+(t) + S_{\alpha/2}^-(t))\phi + \frac{1}{2}(S_{\alpha/2}^+(t)\Psi - \Psi + \Psi - S_{\alpha/2}^-(t)\Psi)\\
%&=& \frac{1}{2}(S_{\alpha/2}^+(t) + S_{\alpha/2}^-(t))\phi + \frac{1}{2}(S_{\alpha/2}^+(t) - S_{\alpha/2}^-(t))\Psi
%\end{eqnarray*}
%the same as (\ref{Fdalembert}).
\end{rem}

By Lemma \ref{decom} we can also derive the following result for more general fractional integro-differential equations.
\begin{thm}\label{Vfdecom}
Let $0 < \alpha \le 2$, $f$ be a continuous function on $X$ which is in $W^{1,1}([0,T],X)$ for every $T>0$, and $A$  a densely defined closed operator on $X$. If the two
Volterra equations
\begin{equation}\label{fvolterra+}
u_{\alpha/2}^+(t) = f(t) + \frac{1}{\Gamma(\alpha/2)}\int_0^t(t-s)^{\alpha/2-1}A u_{\alpha/2}^+(s)ds, \quad t \ge 0
\end{equation}
and
\begin{equation}\label{fvolterra-}
u_{\alpha/2}^-(t) = f(t) - \frac{1}{\Gamma(\alpha/2)}\int_0^t(t-s)^{\alpha/2-1} A u_{\alpha/2}^-(s)ds, \quad t \ge 0
\end{equation}
are well-posed, then the Volterra equation
\begin{equation}\label{fvolterra}
u_{\alpha}(t) = f(t) + \frac{1}{\Gamma(\alpha)}\int_0^t(t-s)^{\alpha-1} A^2 u_{\alpha}(s)ds, \quad t \ge 0
\end{equation}
 is also well-posed. Moreover, the unique mild solution of $(\ref{fvolterra})$ is given by
$$
u_{\alpha}(t) = \frac{1}{2}(u_{\alpha/2}^+(t) +u_{\alpha/2}^-(t)),
$$
where $u_{\alpha/2}^+(t)$ and  $u_{\alpha/2}^-(t)$ are mild solutions of (\ref{fvolterra+}) and (\ref{fvolterra-}), respectively.
\end{thm}

\begin{proof}
The well-posedness of (\ref{fvolterra+}) and (\ref{fvolterra-}) implies that both $A$ and $-A$ generate $\alpha/2$-times resolvent families on $X$, and  we then denote them by $S_{\alpha/2}^+(t)$ and $S_{\alpha/2}^-(t)$ respectively. By Lemma
\ref{decom} $A^2$ generates $\alpha$-times resolvent family $S_\alpha(t)$, and therefore the Volterra equation (\ref{fvolterra}) is also well-posed. By (\ref{Prussf}) we have the mild solution for (\ref{fvolterra+}) and
(\ref{fvolterra-}) are given by
$$
u_{\alpha/2}^+(t)= S_{\alpha/2}^+(t)f(0) + \int_0^t S_{\alpha/2}^+(t-s){f}^\prime(s)ds
$$
and
$$
u_{\alpha/2}^-(t)= S_{\alpha/2}^-(t)f(0) + \int_0^t S_{\alpha/2}^-(t-s){f}^\prime(s)ds,
$$
respectively, and the mild solution for (\ref{fvolterra}) is given by
$$
u_{\alpha}(t)= S_{\alpha}(t)f(0) + \int_0^t S_{\alpha}(t-s){f}^\prime(s)ds,
$$
therefore our claim follows from the fact that $S_\alpha(t) = (S_{\alpha/2}^+(t) + S_{\alpha/2}^-(t))/2$.
\end{proof}

As a consequence of  Theorem \ref{Vfdecom}, we have the next decomposition of the solutions to the following integro-differential equations mentioned in Introduction, and also for their corresponding fractional equations
with Caputo derivatives.

\begin{thm}\label{thmab}
Let $1< \alpha \le 2$, $\alpha/2 \leq \beta \leq \alpha$. Suppose that both $A$ and $-A$ generate $\alpha/2$-times resolvent families $S_{\alpha/2}^+(t)$ and $S_{\alpha/2}^-(t)$ on $X$, respectively.
Then for $\phi, \psi \in X$, the integro-differential equation
\begin{equation} \label{general betaA}
u(t)= \phi + \frac{t^{\beta}}{\Gamma(1+\beta)}\psi + \frac{1}{\Gamma(\alpha)}\int_0^t(t-s)^{\alpha-1}A^2 u(s)ds, \quad t >0
\end{equation}
has a unique mild solution which is given by
\begin{equation}\label{alphabetasolution}
u(t) = \frac{1}{2}(S_{\alpha/2}^+(t) \phi + S_{\alpha/2}^-(t)\phi) + \frac{1}{2}(J_t^\beta S_{\alpha/2}^+(t) \psi + J_t^\beta S_{\alpha/2}^-(t)\psi),
\end{equation}
and $u(t)$  can be decomposed into
$$
u(t) = \frac{1}{2}(u^+(t) +u^-(t)),
$$
where $u^+(t)$ and $u^-(t)$ are mild solutions to
\begin{equation} \label{general betaA+}
u^+(t)= \phi + \frac{t^{\beta}}{\Gamma(1+\beta)}\psi+ \frac{1}{\Gamma(\alpha/2)}\int_0^t(t-s)^{\alpha/2-1}A u^+(s)ds
\end{equation}
and
\begin{equation} \label{general betaA-}
u^-(t)= \phi + \frac{t^{\beta}}{\Gamma(1+\beta)}\psi- \frac{1}{\Gamma(\alpha/2)}\int_0^t(t-s)^{\alpha/2-1}A u^-(s)ds
\end{equation}
respectively. Moreover,  when $\beta=1$, (\ref{alphabetasolution}) is
\begin{equation}\label{firstorder}
u(t) = \frac{1}{2}[S_{\alpha/2}^+(t) \phi + S_{\alpha/2}^- (t)\phi]+\frac{1}{2}\int_0^t [S_{\alpha/2}^+(s) \psi + S_{\alpha/2}^-(s) \psi]ds,
\end{equation}
which gives the unique mild solution to the fractional equation
\begin{equation}\label{ut-equation}
\begin{cases}
D_t^\alpha u(t) = A^2u(t), \quad t>0\\
u(0)= \phi, \quad u_t'(0)= \psi;
\end{cases}
\end{equation}
if $\beta \not=1$, (\ref{alphabetasolution}) gives the unique mild solution to the fractional equation
\begin{equation}\label{alpha-beta}
\begin{cases}
D_t^{\alpha}u(t) = A^2 u(t)+ \frac{t^{\beta-\alpha}}{\Gamma(\beta-\alpha+1)}\psi, \quad t>0\\
u(0)= \phi,\quad  D_t^\beta u(0)= \psi.
\end{cases}
\end{equation}
In both cases $u^+(t)$ and $u^-(t)$ given by (\ref{general betaA+}) and (\ref{general betaA-}), respctively, are mild solutions to
\begin{equation}\label{alpha-beta+}
\begin{cases}
D_t^{\alpha/2}u^+(t) = A u^+(t)+ \frac{t^{\beta-\alpha/2}}{\Gamma(\beta-\alpha/2+1)}\psi, \quad t>0\\
u^+(0)= \phi,
\end{cases}
\end{equation}
and
\begin{equation}\label{alpha-beta-}
\begin{cases}
D_t^{\alpha/2}u^-(t) = -A u^-(t)+ \frac{t^{\beta-\alpha/2}}{\Gamma(\beta-{\alpha \over 2}+1)}\psi, \quad t>0\\
u^-(0)= \phi,
\end{cases}
\end{equation}
respectively. If $\phi \in D(A^2)$ and $\psi \in D(A)$, then both the above mild solutions are strong solutions.
\end{thm}

\begin{proof}
By taking $f(t) = \phi + \frac{t^{\beta}}{\Gamma(1+\beta)}\psi$ in Theorem \ref{Vfdecom}, we have that the mild solution to (\ref{general betaA}) is given by
$$
u(t) = S_\alpha(t) \phi + \int_0^t S_\alpha(t-s) \frac{s^{\beta-1}}{\Gamma(\beta)}\psi ds =  S_\alpha(t) \phi + (J_t^\beta S_\alpha)(t)\psi .
$$
This gives (\ref{alphabetasolution}) since $S_\alpha(t) = \frac{1}{2}(S_{\alpha/2}^+(t) \phi + S_{\alpha/2}^- (t))$. Equations (\ref{ut-equation}) and (\ref{alpha-beta}) follow by differentiating (\ref{general betaA})
$\alpha$-times,
and (\ref{alpha-beta+}), (\ref{alpha-beta-}) hold by differentiating (\ref{general betaA+}), (\ref{general betaA-})   $\alpha/2$-times, respectively.

Now for $\phi \in D(A^2)$ and $\psi \in D(A)$, we need only show that $u(t)$ is a strong solution of (\ref{general betaA}). Indeed, by the resolvent equation (\ref{A-resolventequation}) and Remark \ref{remfirst} we have
\begin{eqnarray*}
u(t) &= & S_\alpha(t) \phi + (J_t^\beta S_\alpha)(t)\psi\\
&=& \phi + (g_\alpha * S_\alpha)(t) A^2 \phi + J_t^\beta(\psi + A^2(g_\alpha* S_\alpha)(t)\psi)\\
&=& \phi + \frac{t^\beta}{\Gamma(\beta +1)} \psi + (g_\alpha*A^2 S_\alpha)(t)\phi +J_t^\beta  A^2(g_\alpha* S_\alpha)(t)\psi.
\end{eqnarray*}
To show that $u$ is a strong solution it remains to show that $(J_t^\beta S_\alpha)(t) \psi \in D(A^2)$ and then $J_t^\beta  A^2 J_t^\alpha  S_\alpha(t)\psi = J_t^\alpha A^2 J_t^\beta S_\alpha(t)\phi$ by the closedness of
the operator $A$. Since $\beta \ge \alpha/2$, by the semigroup properties of Riemann-Liouville integrals and the resolvent equation,
\begin{eqnarray*}
(J_t^\beta S_\alpha)(t) A \psi&=& \frac{1}{2}
(J_t^\beta S_{\alpha/2}^+(t) A \psi + J_t^\beta S_{\alpha/2}^-(t) A\psi)\\
&=& \frac{1}{2}J_t^{\beta-\alpha/2}(J_t^{\alpha/2} S_{\alpha/2}^+(t) A\psi + J_t^{\alpha/2} S_{\alpha/2}^-(t) A\psi)\\
&=& \frac{1}{2}J_t^{\beta-\alpha/2}(S_{\alpha/2}^+(t)\psi -\psi + \psi-  S_{\alpha/2}^-(t) \psi)\\
&=& \frac{1}{2}J_t^{\beta-\alpha/2}(S_{\alpha/2}^+(t)\psi-  S_{\alpha/2}^-(t) \psi),
\end{eqnarray*}
and this implies that $(J_t^\beta S_\alpha)(t) \psi \in D(A^2)$.
\end{proof}
\begin{rem}
(1) When $\alpha=2$, $\beta =1$ and $A= d/dx$ on the line, then
$$
S_{1}^+(t) \phi(x)= \phi(x+t), \quad S_{1}^-(t) \phi(x)= \phi(x-t),
$$
and
$$
\int_0^t [S_{\alpha/2}^+(s) \psi(x) + S_{\alpha/2}^-(s) \psi(x)]ds=\int_0^t [\psi(x+s) + \psi(x-s)]ds=  \int_{x-t}^{x+t} \psi(y)dy ,
$$
so the formula (\ref{alphabetasolution}) is exactly the classical d'Alembert's formula (\ref{dAlmbert}). It is therefore reasonable to call
(\ref{alphabetasolution}) the d'Alembert formula for (\ref{general betaA}).

(2) When $\beta =1$, the equation (\ref{general betaA}) becomes
\begin{equation}\label{tdAlembert}
u(t)= \phi + t\psi + \frac{1}{\Gamma(\alpha)}\int_0^t(t-s)^{\alpha-1}A^2u(s)ds.
\end{equation}
 The identity
for $u$ in (\ref{firstorder}) could be considered as the d'Alembert formula for
Volterra equation (\ref{tdAlembert}).

(3) When $\beta = \alpha/2$, then (\ref{general betaA}) is in fact (\ref{Fujitaabstract}). The formula
$$
u(t) = \frac{1}{2}(S_{\alpha/2}^+(t) \phi + S_{\alpha/2}^-(t)\phi) + \frac{1}{2}(J_t^{\alpha/2} S_{\alpha/2}^+(t) \psi + J_t^{\alpha/2} S_{\alpha/2}^-(t)\psi)
$$
is then an alternative d'Alembert formula for (\ref{Fujitaabstract}), and a generalization of (\ref{Fujitaformula}). Their difference will be clarified in Remark \ref{difference}.

(4) The $\alpha$-order equation
\begin{equation*}
\begin{cases}
D_t^\alpha u(t) = A^2 u(t), \quad t>0\\
u(0)= \phi, \quad u_t'(0)= 0
\end{cases}
\end{equation*}
is equivalent to the sequential $\alpha$-order equation
\begin{equation*}
\begin{cases}
D_t^{\alpha/2}D_t^{\alpha/2} u(t) = A^2 u(t), \quad t>0\\
u(0)= \phi, \quad D_t^{\alpha/2}(0)= 0;
\end{cases}
\end{equation*}
their corresponding integral equation is
\begin{equation*}
u(t)= \phi +  \frac{1}{\Gamma(\alpha)}\int_0^t(t-s)^{\alpha-1}A^2u(s)ds.
\end{equation*}
The solution is given by $u(t) = \frac{1}{2}(S_{\alpha/2}^+(t) \phi + S_{\alpha/2}^-(t)\phi)$.
\end{rem}

On the other hand, motivated by the analysis in the Introduction, we connect the integro-differential
equation (\ref{general betaA}) with a fractional differential equation of
sequential Caputo derivatives \cite{Podlubny}.

\begin{thm}\label{afie}
Let $1< \alpha \le 2$ and $\alpha/2 \leq \beta \leq \alpha$. Suppose that both $A$ and $-A$ generate $\alpha/2$-times resolvent families $S_{\alpha/2}^+(t)$ and $S_{\alpha/2}^-(t)$ on $X$, respectively.  Suppose also that
$\phi \in D(A^2)$
and $\psi  \in D(A)$. Then (\ref{alphabetasolution}) gives
a strong solution to the sequential $\alpha$-order Cauchy problem
\begin{equation}\label{two-equation2}
\begin{cases}
D_t^{\alpha- \beta} D_t^{\beta} u(t) = A^2 u(t), \quad t>0\\
u(0)= \phi, \quad D_t^{\beta}u(0)= \psi.
\end{cases}
\end{equation}
\end{thm}

\begin{proof}
Note that for $\phi \in D(A^2)$ by applying the resolvent equation (\ref{A-resolventequation}) twice we have
\begin{eqnarray*}
S_{\alpha/2}^\pm(t) \phi &=& \phi \pm (g_{\alpha/2}* S_{\alpha/2}^\pm)(t)A \phi \\
&=& \phi \pm (g_{\alpha/2}*1) (t)A\phi + (g_{\alpha/2}* g_{\alpha/2}* S_{\alpha/2}^\pm)(t)A^2 \phi\\
&=&\phi \pm g_{\alpha/2+1} (t)A\phi + (g_{\alpha}* S_{\alpha/2}^\pm)(t)A^2 \phi,
\end{eqnarray*}
thus $S_{\alpha/2}^\pm(t) \phi$ are $\beta$-order differentiable and
$$
D_t^\beta S_{\alpha/2}^\pm (t) \phi = \pm g_{\alpha/2+1-\beta} (t)A\phi + (g_{\alpha-\beta}* S_{\alpha/2}^\pm)(t)A^2 \phi .
$$
By summing the above two identities we get
$$
D_t^\beta (S_{\alpha/2}^+ (t) \phi +S_{\alpha/2}^- (t) \phi) = g_{\alpha-\beta}*(S_{\alpha/2}^+ +  S_{\alpha/2}^-)(t)A^2 \phi ,
$$
from which it follows that $D_t^\beta (S_{\alpha/2}^+ (t) \phi +S_{\alpha/2}^- (t) \phi)$ is $(\alpha-\beta)$-order differentiable and
\begin{equation}\label{Eq1}
\begin{split}
D_t^{\alpha - \beta} D_t^\beta (S_{\alpha/2}^+ (t) \phi +S_{\alpha/2}^- (t) \phi) &= (S_{\alpha/2}^+ +  S_{\alpha/2}^-)(t)A^2 \phi \\
&= A^2(S_{\alpha/2}^+ (t) \phi+  S_{\alpha/2}^-(t)\phi) ,
\end{split}
\end{equation}
since $A$ commutes with $S_{\alpha/2}^\pm (t)$. For $\psi \in D(A)$, as in the proof of Theorem \ref{thmab}
% noting that $\beta \ge \alpha/2$, by the resolvent equation (\ref{A-resolventequation}),
%\begin{eqnarray*}
%A (J_t^\beta S_{\alpha/2}^+(t) \psi + J_t^\beta S_{\alpha/2}^-(t) \psi)&=& J_t^\beta S_{\alpha/2}^+(t) A\psi + J_t^\beta S_{\alpha/2}^-(t) A\psi \\
%&=& J_t^{\beta-\alpha/2}(J_t^{\alpha/2} S_{\alpha/2}^+(t) A\psi + J_t^{\alpha/2} S_{\alpha/2}^-(t) A\psi)\\
%&=& J_t^{\beta-\alpha/2}(S_{\alpha/2}^+(t)\psi -\psi + \psi-  S_{\alpha/2}^-(t) \psi)\\
%&=& J_t^{\beta-\alpha/2}(S_{\alpha/2}^+(t)\psi-  S_{\alpha/2}^-(t) \psi)
%\end{eqnarray*}
we have $(J_t^\beta S_{\alpha/2}^+(t) \psi + J_t^\beta S_{\alpha/2}^-(t) \psi) \in D(A^2)$  and
\begin{equation}\label{Eq2}
A^2 (J_t^\beta S_{\alpha/2}^+(t) \psi + J_t^\beta S_{\alpha/2}^-(t) \psi) = J_t^{\beta-\alpha/2}(S_{\alpha/2}^+(t)A\psi-  S_{\alpha/2}^-(t) A\psi).
\end{equation}
On the other hand, since
$$
J_t^\beta S_{\alpha/2}^\pm (t) \psi = g_{\beta+1}(t) \psi \pm J_t^{\beta+ \alpha/2}S_{\alpha/2}^+(t) A\psi,
$$
we have
$$
D_t^\beta J_t^\beta S_{\alpha/2}^\pm (t) \psi = \psi \pm J_t^{\alpha/2}S_{\alpha/2}^+(t) A\psi = S_{\alpha/2}^\pm \psi,
$$
and thus
\begin{equation}\label{Eq3}
D_t^{\alpha - \beta}D_t^\beta (J_t^\beta S_{\alpha/2}^+ (t) \psi +J_t^\beta S_{\alpha/2}^- (t) \psi)= J_t^{\beta-\alpha/2}(S_{\alpha/2}^+(t)A\psi-  S_{\alpha/2}^-(t) A\psi).
\end{equation}
Combining (\ref{Eq1}), (\ref{Eq2}) and (\ref{Eq3}) we have proven that $u(t)$ defined by (\ref{alphabetasolution}) is a strong solution of (\ref{two-equation2}), where
the initial conditions are easy to check.
\end{proof}

\begin{rem}\label{difference}
As a direct consequence, we have  the following sequential fractional differential equation corresponding to (\ref{Fujitaabstract}):
\begin{equation*}
\begin{cases}
D_t^{\alpha/2} D_t^{\alpha/2} u(t) = A^2 u(t), \quad t>0\\
u(0)= \phi, \quad D_t^{\alpha/2}u(0)= \psi.
\end{cases}
\end{equation*}
The solution can be decomposed as half of the sum of the solutions to
\begin{equation*}
\begin{cases}
D_t^{\alpha/2}u^+(t) = A u^+(t)+ \psi, \quad t>0\\
u^+(0)= \phi,
\end{cases}
\end{equation*}
and
\begin{equation*}
\begin{cases}
D_t^{\alpha/2}u^-(t) = -A u^-(t)+ \psi, \quad t>0\\
u^-(0)= \phi.
\end{cases}
\end{equation*}
In case of $A = \partial/\partial x$ and $\alpha=2$, this means that we can decompose the
solution of the traditional wave equation (\ref{wave-equation}) into the half sum of solutions to the following two first-order nonhomogeneous equation
\begin{equation*}
\begin{cases}
u_t^+(t,x)=u_x^+(t,x) + \psi(x)\\
u^+_0(t,x) = \phi(x)
\end{cases}
\mbox{ and  }
\begin{cases}
u_t^-(t,x)=-u_x^-(t,x) + \psi(x)\\
u_0^-(t,x) = \phi(x).
\end{cases}
\end{equation*}
In other words, Fujita's decomposition means that one can decompose the solution of the wave equation into the half sum of
$$
u^+(x,t) = \phi(x+t)+ \int_0^{x+t}\psi(y)dy\quad \mbox{  and  }\quad u^-(x,t) = \phi(x-t)- \int_0^{x-t}\psi(y)dy;
$$
while our decomposition is
$$
u^+(x,t) = \phi(x+t)+ \int_x^{x+t}\psi(y)dy\quad \mbox{  and  }\quad u^-(x,t) = \phi(x-t)+ \int_{x-t}^{x}\psi(y)dy.
$$
\end{rem}
\mbox{}\\
%\begin{proof}
% From the proof of Theorem \ref{Vfdecom} we know that the solution of (\ref{Fujitaabstract}) is given by
 %$$
 %u(t) = S_\alpha(t)\phi + \int_0^t S_\alpha(t-s) \frac{s^{\frac{\alpha}{2}-1}}{\Gamma(\frac{\alpha}{2})}\psi ds.
 %$$
% Since $S_\alpha(t) = \frac{1}{2}[S_{\alpha/2}^+(t)  + S_{\alpha/2}^- (t)]$, (\ref{Fdalembert}) follows immediately.
%\end{proof}

%As a direct consequence of the Lemma \ref{decom}, we have the following result:
%\begin{thm}
%Let $0< \alpha \le 2$ and $\phi \in X$. If the two $\alpha/2$-order Cauchy problems
%\begin{equation}\label{0+equation}
%\begin{cases}
%&D_t^{\alpha/2}u(t) = Au(t), \quad t>0\\
%&u(0)= \phi
%\end{cases}
%\end{equation}
%and
%\begin{equation}\label{0-equation}
%\begin{cases}
%&D_t^{\alpha/2}u(t) = -Au(t), \quad t>0\\
%&u(0)= \phi
%\end{cases}
%\end{equation}
%are well-posed, then the $\alpha$-order Cauchy problem
%\begin{equation}\label{alpha-equation}
%\begin{cases}
%&D_t^\alpha u(t) = A^2u(t), \quad t>0\\
%&u(0)= \phi, \quad (u_t(0)= 0 \mbox{ if } 1<\alpha \le 2)
%\end{cases}
%\end{equation}
%is also well-posed. Moreover, the unique mild solution $u_\alpha$ of $(\ref{alpha-equation})$ is given by
%$$
%u_\alpha(t) = \frac{1}{2}[u_{\alpha/2}^+(t) + u_{\alpha/2}^-(t)],
%$$
%where $u_{\alpha/2}^+(t)$ and $u_{\alpha/2}^-(t)$ are the solutions of $(\ref{0+equation})$ and $(\ref{0-equation})$ respectively.
%\end{thm}

Finally in this section, let us consider the Riemann-Liouville fractional differential equation
\begin{equation}\label{RL-equation}
\begin{cases}
{}_{RL}D_t^{\alpha}u(t) = A^2 u(t), \quad t>0\\
(g_{2-\alpha}*u)(0)= 0,\quad  (g_{2-\alpha} *u)'(0)= \psi
\end{cases}
\end{equation}
with $1< \alpha \le 2$. By integration with respect to $t$ for $\alpha$-times, we get the corresponding integro-differential equation as follows
$$
u(t) = \frac{t^{\alpha-1}}{\Gamma(\alpha)}\psi + \frac{1}{\Gamma(\alpha)}\int_0^t (t-s)^{\alpha-1}A^2u(s)ds.
$$
If $A^2$ generates an $\alpha$-times resolvent family $S_\alpha(t)$, then the solution $u$ is given by $J_t^{\alpha-1}S_\alpha(t)\psi$ \cite{LiPeng}. Indeed, it is easy to verify that
\begin{eqnarray*}
u(t)&= & J_t^{\alpha-1}S_\alpha(t)\psi\\
&=& J_t^{\alpha-1}(\psi + A^2(g_\alpha * S_\alpha)\psi) \\
&=& g_\alpha(t) \psi + A^2(g_\alpha* J_t^{\alpha-1}S_\alpha)\psi\\
&=& g_\alpha(t) \psi + A^2(g_\alpha*u)(t).
\end{eqnarray*}
On the other hand,
$$
u(t)= J_t^{\alpha-1}S_\alpha(t)\psi = \frac{1}{2}J_t^{\alpha-1} (S_{\alpha/2}^+(t) + S_{\alpha/2}^-(t))\psi
$$
where $J_t^{\alpha-1} S_{\alpha/2}^+(t) \psi$ and $J_t^{\alpha-1}  S_{\alpha/2}^-(t)\psi$ are solutions of
$$
v(t) = \frac{t^{\alpha-1}}{\Gamma(\alpha)}\psi + \frac{1}{\Gamma(\alpha/2)}\int_0^t (t-s)^{\alpha/2-1}A v(s)ds.
$$
and
$$
v(t) = \frac{t^{\alpha-1}}{\Gamma(\alpha)}\psi - \frac{1}{\Gamma(\alpha/2)}\int_0^t (t-s)^{\alpha/2-1}Av (s)ds.
$$
respectively, and the corresponding Riemann-Liouville fractional equations are
\begin{equation}\label{RL-equation+}
\begin{cases}
{}_{RL}D_t^{\alpha/2}v(t) = A v(t) + \frac{t^{\alpha/2-1}}{\Gamma(\alpha/2)}\psi, \quad t>0\\
(g_{1-\alpha/2}*v)(0)= 0
\end{cases}
\end{equation}
and
\begin{equation}\label{RL-equation-}
\begin{cases}
{}_{RL}D_t^{\alpha/2}v(t) = - A v(t) + \frac{t^{\alpha/2-1}}{\Gamma(\alpha/2)}\psi, \quad t>0\\
(g_{1-\alpha/2}*v)(0)= 0
\end{cases}
\end{equation}
respectively.

In summary, we have
\begin{thm}
Let $1< \alpha \le 2$. If both $A$ and $-A$ generate $\alpha/2$-times resolvent families $S_{\alpha/2}^+(t)$ and $S_{\alpha/2}^-(t)$ on $X$, respectively, then the unique mild solution of the $\alpha$-order
Riemann-Liouville fractional Cauchy problem (\ref{RL-equation}) with $\psi \in X$
is given by
\begin{equation}
u_\alpha(t) = \frac{1}{2}(J_t^{\alpha-1} S_{\alpha/2}^+(t)\psi + J_t^{\alpha-1}  S_{\alpha/2}^-(t)\psi)=: \frac{1}{2}(u_{\alpha/2}^+(t)
+u_{\alpha/2}^-(t)),
\end{equation}
where $u_{\alpha/2}^+(t)$ and $u_{\alpha/2}^-(t)$ are mild solution to (\ref{RL-equation+}) and (\ref{RL-equation-}), respectively.
\end{thm}

\section{Examples}
In this section, we will illustrate the interpretation and application of our abstract setting
with several examples. The inverse of a standard stable subordinator will play an important role in our examples,
therefore we will give a brief of the connection with
time-fractional differential equation for reader's convenience. See \cite{meerschaert2012book,meerschaert2013} for more details.

%We will use $Af(x)=f^\prime (x)$ in {\bf Examples 3.1, 3.2} and {\bf 3.3}, while
%$Af(x)=f(x-1)-f(x)$ in {\bf Examples 3.4} and {\bf 3.5}.

Throughout this section, we define $S_t$ as the following processes on $\{x:x\geq0\}$ for $0<\alpha\leq2$:
\begin{itemize}
\item[(i)]
 When $0<\alpha<2$, $S_t$ is a standard $\alpha/2$-stable subordinator, a L\'evy process such that ${\bf E} \exp\{-\lambda S_t\}= \exp\{-t \lambda^{\alpha/2}\}$;
\item[(ii)]
 When $\alpha=2$, $S_t$ is a deterministic continuous process with uniform velocity $1$, starting from $0$. Indeed, $S_t$
in this case can be regarded as a degenerate stable subordinator with density function $\delta(x-t)$ and ${\bf E} \exp\{-\lambda S_t\}= \exp\{-t \lambda\}$. For unity, we will use $1$-stable subordinator to denote $S_t$ for $\alpha=2$.
\end{itemize}
Then the inverse $\alpha/2$-stable subordinator $Y_{\alpha/2}(t)=\inf\{u>0:S_u>t\}$ is a continuous, nondecreasing and nonnegative process
with $Y_{\alpha/2}(0)=0$, and
$$
{\bf E} \exp\{-s Y_{\alpha/2} (t)\}= E_{\alpha/2}(-st^{\alpha/2}),
$$
where $E_\alpha(t) = \sum_{n=0}^\infty \frac{t^n}{\Gamma(\alpha n +1)}$ is the Mittag-Leffler function.
Denote the probability density function of $Y_{\alpha/2}(t)$ by $\varphi_{\alpha/2}(t,\cdot)$, then we have
\begin{equation}\label{stablepdf}
\int_0^\infty e^{-\lambda t}\varphi_{\alpha/2} (t,s)dt= \lambda^{{\alpha\over 2} -1} e^{-s \lambda^{\alpha/2}} .
\end{equation}

From the above identity, it is not hard to
show that for a suitable function $f$ the following holds for $\lambda$ large enough:
\begin{equation}\label{Yalpha}
\int_0^\infty e^{-\lambda t} {\bf E} [f(Y_{\alpha/2}(t))] dt = \lambda^{{\alpha \over 2} -1} \int_0^\infty f(s) e^{-s\lambda^{\alpha/2}} ds.
\end{equation}
On the other hand, since the first-order derivative operator $\partial_x f(t,x) = f_x(t,x)$ generates a $C_0$-semigroup given by
$T(t)\phi(x) = \phi(x+t)$, by the subordination principle for fractional resolvent family \cite{Baj}, for $1<\alpha <2$, the $\alpha/2$-times resolvent family $S_{\alpha/2}^+(t)$ is given by
$$
S_{\alpha/2}^+(t)\phi(x) = \int_0^\infty \varphi_{\alpha/2}(t,s) T(s)\phi(x) ds=  \int_0^\infty \varphi_{\alpha/2}(t,s)\phi(s+x) ds,
$$
It follows that
$$
\int_0^\infty e^{-\lambda t} S_{\alpha/2}^+(t)\phi(x)dt = \lambda^{{\alpha \over 2} -1} \int_0^\infty \phi(s+x) e^{-s\lambda^{\alpha/2}} ds
$$
from (\ref{stablepdf}).

Comparing with (\ref{Yalpha}) and using the uniqueness of the Laplace transform, we have
\begin{equation}\label{EYalpha+}
 S_{\alpha/2}^+(t)\phi(x) = {\bf E} [\phi(x+Y_{\alpha/2}(t))].
\end{equation}
This is to say that ${\bf E} [\phi(x+Y_{\alpha/2}(t))]$ gives the unique solution to the fractional differential equation
\begin{equation}\label{negative FDE of inverse subordintor}
\begin{cases}
D_t^{\alpha/2}u(t,x) = u_x(t,x), \quad t>0, \quad x\in\Real \\
u(0,x)= \phi(x).
\end{cases}
\end{equation}
Similarly,  the unique solution to the fractional differential equation
\begin{equation}\label{FDE of inverse subordintor}
\begin{cases}
D_t^{\alpha/2}u(t,x) = -u_x(t,x), \quad t>0, \quad x\in\Real \\
u(0,x)= \phi(x)
\end{cases}
\end{equation}
is given by
\begin{equation}\label{EYalpha-}
 S_{\alpha/2}^-(t)\phi(x) = {\bf E} [\phi(x-Y_{\alpha/2}(t))].
\end{equation}
The fundamental solution to equation (\ref{FDE of inverse subordintor})
is the probability  density function of an inverse $\alpha/2$-stable subordinator. Meanwhile, the fundamental solution to
equation (\ref{negative FDE of inverse subordintor})
is the reflected density ($x\mapsto -x$) of an inverse $\alpha/2$-stable subordinator.
%The stable subordinator is a stable L{\'e}vy process with non-decreasing path,
%while the inverse of stable subordinator is certainly non-decreasing.

%In \textbf{Subsection 3.1}, we will consider the stochastic process interpretation for the d'Alembert solution
%to
%\begin{equation}
%u(t,x)= \phi(x) + \frac{1}{\Gamma(\alpha)}\int_0^t(t-s)^{\alpha-1}\partial_x^2u(s,x)ds
%\end{equation}

%Then we will study the stochastic interpretation for
%\begin{equation}
%u(t,x)= \phi(x) + \frac{t^{\beta}}{\Gamma(1+\beta)}\psi(x) + \frac{1}{\Gamma(\alpha)}\int_0^t(t-s)^{\alpha-1}\partial_x^2u(s,x)ds
%\end{equation}
%in \textbf{Subsection 3.2}.

%Finally, we will compare equation
%\begin{equation}
%u(t,x)= \phi(x) + \frac{t^{\alpha/2}}{\Gamma(1+{\alpha\over 2})}\psi(x) + \frac{1}{\Gamma(\alpha)}\int_0^t(t-s)^{\alpha-1}\partial_x^2u(s,x)ds
%\end{equation}
%with
%\begin{equation}
%v(t,x)= \phi(x) + t\psi(x) + \frac{1}{\Gamma(\alpha)}\int_0^t(t-s)^{\alpha-1}\partial_x^2v(s,x)ds.
%\end{equation}
%in \textbf{Subsection 3.3}.

\begin{exm}\label{exm1}

%There are at least three nice stochastic processes can be used to subordinate. The first is the $Y_{\alpha/2} (t)$ in \cite{Fujita}.
%See Meerschaert's paper for the other stochastic processes. I will try to explain the d'Alembert solutions further.

Let $1<\alpha\leq 2$. By Theorem \ref{abFujita} and the above analysis, the solution of
\begin{equation}\label{exm1equation}
u(t,x)= \phi(x) + \frac{t^{\alpha/2}}{\Gamma(1+ \alpha/2)}\psi(x)+\frac{1}{\Gamma(\alpha)}\int_0^t(t-s)^{\alpha-1}u_{xx}(s,x)ds
\end{equation}
or
\begin{equation}\label{3.3}
\begin{cases}
D_t^{\alpha/2}D_t^{\alpha/2}u(t,x) =  u_{xx}(t,x), \quad t>0, \quad x\in\Real \\
u(0,x)= \phi(x), \quad D_t^{\alpha/2}u(0,x)=\psi(x)
\end{cases}
\end{equation}
can be represented as
\begin{equation}\label{erp}
u(t,x) = \frac{1}{2}{\bf E}[\phi(x+ Y_{\alpha/2}(t)) + \phi(x-Y_{\alpha/2}(t))] + \frac{1}{2}{\bf E} \int_{x-Y_{\alpha/2}(t)}^{x+Y_{\alpha/2}(t)} \psi(y)dy.
\end{equation}
This is the same as Fujita's formula (\ref{Fujitaformula}). Thus the motion (or the evolution) of $u(t,x)$ governed by equation (\ref{3.3}) is mixed with two
components: positive direction and negative direction on line $\Real$.
The solution of
\begin{equation*}
\begin{cases}
D_t^{\alpha/2}u(t,x) = u_x(t,x), \quad t>0, \quad x\in\Real \\
u(0,x)= \phi(x) + \int_0^x \psi(y)dy
\end{cases}
\end{equation*}
 gives the motion on positive direction,
while the solution of equation
\begin{equation*}
\begin{cases}
D_t^{\alpha/2}u(t,x) = -u_x(t,x), \quad t>0, \quad x\in\Real \\
u(0,x)= \phi(x) - \int_0^x \psi(y)dy
\end{cases}
\end{equation*}
 gives the motion on negative direction.
\end{exm}

The above interpretation, adapted from \cite{meerschaert}, is consistent with the classical d'Alembert formula for the one-dimensional
wave equation on line: decompose the wave equation into two transport equations with positive direction and negative direction.
Moreover, the fractional d'Alembert formula indicates the `random alternative of direction' or `wander
in positive or negative direction' for the motion governed by one dimensional fractional wave equation \cite{gorenflo2015}.

\begin{exm}\label{exm2}
Here we consider the following integro-differential equation
\begin{equation}\label{general beta}
u(t,x)= \phi(x) + \frac{t^{\beta}}{\Gamma(1+\beta)}\psi(x) + \frac{1}{\Gamma(\alpha)}\int_0^t(t-s)^{\alpha-1}u_{xx}(s,x)ds
\end{equation}
and the corresponding fractional Cauchy problem
\begin{equation}
\begin{cases}
D_t^{\alpha-\beta} D_t^{\beta} u(t,x) = u_{xx}(t,x), \quad t>0 \\
u(0,x)= \phi(x), \quad D_t^{\beta}u(0,x)= \psi(x)
\end{cases}
\end{equation}
where $\alpha/2 \leq \beta \leq \alpha$.
%(Note: the restriction $\alpha/2 \leq \beta \leq \alpha$ is from equation (5.2) and $g_{\beta-\frac{\alpha}{2}}\ast f$ in the following computation.)
We will explore the connection between their fractional d'Alembert solution and its stochastic interpretation.

To start with, we first claim the following formula holds:
\begin{equation}\label{fractional integral of PDF}
(g_{\beta}\ast f_{\alpha/2})(t,y)=\int_y^{+\infty}(g_{\beta-\alpha/2}\ast f_{\alpha/2})(t,z)dz,  \quad t,y>0.
\end{equation}
where $f_{\alpha/2}(t,y)$ is the probability density function of the inverse $\alpha/2$-stable subordinator $Y_{\alpha/2}(t)$.
By the properties of the inverse stable subordinator \cite{meerschaert2012book},
 the Laplace transform on variable $t$ of the left side is
$$
\widehat{g_{\beta}}(\lambda)\widehat{f_{\alpha/2}}(\lambda,y)=\lambda^{-\beta}\cdot \lambda^{\alpha/2-1}e^{-\lambda^{\alpha/2}y};
$$
on the other hand, the Laplace transform on variable $t$ of the right side is
\begin{eqnarray*}
&&\int_0^{+\infty}e^{-\lambda t}dt \int_y^{+\infty}(g_{\beta-\alpha/2}\ast f_{\alpha/2})(t,z)dz \\
&=&  \int_y^{+\infty} dz \int_0^{+\infty}e^{-\lambda t} (g_{\beta-\alpha/2}\ast f_{\alpha/2})(t,z) dt\\
&=&  \int_y^{+\infty} \lambda^{-\beta+\alpha/2}\cdot \lambda^{\alpha/2-1}e^{-\lambda^{\alpha/2}z} dz \\
%&=& \lambda^{-\beta}\cdot \lambda^{\frac{\alpha}{2}-1} \big[ -e^{-\lambda^{\alpha/2}z} \big]\big|_{z=y}^{z=\infty} \\
&=& \lambda^{-\beta}\cdot \lambda^{\alpha/2-1}e^{-\lambda^{\alpha/2}y}.
\end{eqnarray*}
Therefore we draw the conclusion by the uniqueness of the Laplace transform.

Also, the equation (\ref{fractional integral of PDF}) is formally equivalent to
$$
D^{\alpha/2}_t(g_{\beta-\alpha/2}\ast f_{\alpha/2})(t,y)=-\frac{d}{dy}(g_{\beta-\alpha/2}\ast f_{\alpha/2})(t,y),
$$
which reduces to the governing equation for the inverse $\alpha/2$-stable subordinator when $\beta=\alpha/2$.
The term $(g_{\beta-\alpha/2}\ast f_{\alpha/2})(t,y) $ is not probability density if $\beta\neq \alpha/2$.
However, for any $\alpha/2\leq \beta\leq \alpha$, it is easy to see $(g_{\beta-\alpha/2}\ast f_{\alpha/2})(t,y) $ is positive and could be normalized as
$(g_{1+\beta-\alpha/2}(t))^{-1}(g_{\beta-\alpha/2}\ast f_{\alpha/2})(t,y)$,
i.e., $\int_0^\infty (g_{1+\beta-\alpha/2}(t))^{-1}(g_{\beta-\alpha/2}\ast f_{\alpha/2})(t,y)dy=1$. We denote the corresponding random variables
by $H_{\beta}(t)$, and note that $H_{\beta}(0)=0$ almost surely.
%This can be shown as follows:
%\begin{eqnarray*}
%&&\int_0^\infty (g_{1+\beta-\frac{\alpha}{2}}(t))^{-1}(g_{\beta-\frac{\alpha}{2}}\ast f)(t,y)dy \\
%&=& \int_0^\infty (g_{1+\beta-\frac{\alpha}{2}}(t))^{-1}\int_0^t g_{\beta-\frac{\alpha}{2}}(t-s) f(s,y)dsdy \\
%&=& (g_{1+\beta-\frac{\alpha}{2}}(t))^{-1}\int_0^t \int_0^\infty g_{\beta-\frac{\alpha}{2}}(t-s) f(s,y)dyds \\
%&=& (g_{1+\beta-\frac{\alpha}{2}}(t))^{-1}\int_0^t  g_{\beta-\frac{\alpha}{2}}(t-s) ds \\
%&=& (g_{1+\beta-\frac{\alpha}{2}}(t))^{-1}g_{1+\beta-\frac{\alpha}{2}}(t)\\
%&=&1
%\end{eqnarray*}
%for all $t\geq0$. And $(g_{1+\beta-\frac{\alpha}{2}}(t))^{-1}(g_{\beta-\frac{\alpha}{2}}\ast f)(t,y)$
%is clearly positive.

Let $S_{\alpha/2}^+(t)$ and $S_{\alpha/2}^-(t)$ be given by (\ref{EYalpha+}) and (\ref{EYalpha-}), respectively.
By Theorem \ref{thmab}, the solution of (\ref{general beta}) can be represented as
$$
u(t,x)= \frac{1}{2}(S_{\alpha/2}^+(t) + S_{\alpha/2}^-(t))\phi(x)
  + \frac{1}{2}(J_t^\beta S_{\alpha/2}^+(t)\psi(x) + J_t^\beta S_{\alpha/2}^-(t) \psi(x)).
$$
By formula (\ref{fractional integral of PDF}) we have
\begin{eqnarray*}
J_t^\beta S_{\alpha/2}^+(t)\psi(x)
&=&\int_0^t g_\beta(t-s) {\bf E} [\psi(x+Y_{\alpha/2}(s))]ds\\
&=&\int_0^t g_\beta(t-s) \int_0^\infty \psi(x+y)f_{\alpha/2}(s,y)dyds \\
&=&\int_0^\infty  \psi(x+y)\int_0^t g_\beta(t-s)f_{\alpha/2}(s,y)dsdy\\
&=&\int_0^\infty  \psi(x+y)\int_y^{\infty}(g_{\beta-\alpha/2}\ast f_{\alpha/2})(t,z)dzdy \\
&=&\int_x^\infty \psi(y)dy \int_{y-x}^{\infty} (g_{\beta-\alpha/2}\ast f_{\alpha/2})(t,z) dz \\
&=&\int_0^\infty (g_{\beta-\alpha/2}\ast f_{\alpha/2})(t,z) dz \int_{x}^{x+z} \psi(y)dy  \\
%&=&\int_0^\infty \int_{0}^{x+z} \psi(y)dy (g_{\beta-\frac{\alpha}{2}}\ast f_{\alpha/2})(t,z) dz\\
%&&\quad - \int_0^\infty \int_{0}^{x} \psi(y)dy (g_{\beta-\frac{\alpha}{2}}\ast f_{\alpha/2})(t,z) dz \\
%&=&g_{1+\beta-\frac{\alpha}{2}}(t)\int_0^\infty \int_{0}^{x+z} \psi(y)dy \big [(g_{1+\beta-\frac{\alpha}{2}}(t))^{-1}(g_{\beta-\frac{\alpha}{2}}\ast f_{\alpha/2})(t,z) \big ]dz \\
%&&\quad - \int_0^\infty \int_{0}^{x} \psi(y)dy (g_{\beta-\frac{\alpha}{2}}\ast f_{\alpha/2})(t,z) dz\\
&=&g_{1+\beta-\alpha/2}(t){\bf E}[\Psi(x+H_{\beta}(t))]- \int_0^\infty \Psi(y) (g_{\beta-\alpha/2}\ast f_{\alpha/2})(t,z) dz
\end{eqnarray*}
where $\Psi(x) = \int_0^x \psi(y)dy$; similarly we have
$$
J_t^\beta S_{\alpha/2}^-(t)\psi(x)
%&=&\int_0^t g_\beta(t-s) {\bf E} [\psi(x-Y_{\alpha/2}(s))]ds\\
%&=&\int_0^\infty \int^{0}_{x-z} \psi(y)dy (g_{\beta-\frac{\alpha}{2}}\ast f_{\alpha/2})(t,z) dz
%+ \int_0^\infty \int_{0}^{x} \psi(y)dy (g_{\beta-\frac{\alpha}{2}}\ast f_{\alpha/2})(t,z) dz\\
=-g_{1+\beta-\alpha/2}(t){\bf E}[\Psi(x-H_{\beta}(t))]+ \int_0^\infty \Psi(x) (g_{\beta-\alpha/2}\ast f_{\alpha/2})(t,z) dz.
$$
Therefore,
\begin{eqnarray*}
u(t,x)
&=& \frac{1}{2}(S_{\alpha/2}^+(t) + S_{\alpha/2}^-(t))\phi(x)\\
&&\quad + \frac{1}{2} g_{1+\beta-\alpha/2}(t)({\bf E}[\Psi(x+H_{\beta}(t))]
     -{\bf E}[\Psi(x-H_{\beta}(t))] )\\
&=&\frac{1}{2}{\bf E}[\phi(x+ Y_{\alpha/2}(t)) + \phi(x-Y_{\alpha/2}(t))]+
   \frac{1}{2}g_{1+\beta-\alpha/2}(t) {\bf E}\int_{x-H_{\beta}(t)}^{x+H_{\beta}(t)}\psi(y)dy.
\end{eqnarray*}

Roughly speaking, the fractional d'Alembert's formula is essentially based on the properties of the square root of
$A^2f(x)=f^{\prime\prime}(x)$,
rather than the order of the initial values. The formula and stochastic interpretation are clean and elegant
when $\beta=\alpha/2$, because $H_{\beta}(t)=Y_{\alpha/2}(t)$ and $g_{1+\beta-\alpha/2}(t)=1$ in this case. The above formula is then consistent with the formula (\ref{Fujitaformula}) of
Fujita.
However, we are unable to clarify the interpretation for the other $ H_{\beta}(t)$.

\end{exm}

As a continuation of Example \ref{exm2}, we shall compare two special cases: $\beta=\alpha/2$ and $\beta=1$.
\begin{exm}\label{exm3}
The solution to (\ref{exm1equation})
%\begin{equation}
%u(t,x)= \phi(x) + \frac{t^{\alpha/2}}{\Gamma(1+{\alpha\over 2})}\psi(x) + \frac{1}{\Gamma(\alpha)}\int_0^t(t-s)^{\alpha-1}u_{xx}(s,x)ds.
%\end{equation}
%and differential equation
%\begin{equation}
%\begin{cases}
%D_t^{\alpha/2} D_t^{\alpha/2} u(t,x) = \partial_x^2u(t,x), \quad t>0 \\
%u(0,x)= \phi(x), \quad \lim_{t \to 0}D_t^{\alpha/2}u(t,x)= \psi(x)
%\end{cases}
%\end{equation}
%are equivalent formally. Their solution
could be represented in two different ways:
\begin{eqnarray*}
u(t,x)
=\frac{1}{2}{\bf E}[\phi(x+ Y_{\alpha/2}(t)) + \phi(x-Y_{\alpha/2}(t))] + \frac{1}{2}{\bf E} \int_{x-Y_{\alpha/2}(t)}^{x+Y_{\alpha/2}(t)} \psi(y)dy
\end{eqnarray*}
and
\begin{eqnarray*}
u(t,x)
&=& \frac{1}{2}{\bf E}[\phi(x+ Y_{\alpha/2}(t)) + \phi(x-Y_{\alpha/2}(t))]   \\
&& + \frac{1}{2}\big \{\frac{1}{\Gamma(\alpha/2)}\int_0^t(t-s)^{\alpha/2-1}{\bf E} [\psi(x+Y_{\alpha/2}(s))+\psi(x-Y_{\alpha/2}(s))]ds\big\}.
\end{eqnarray*}
%where
%$$
%{\bf E} \int_0^{x+Y_{\alpha/2}(t)} \psi(y)dy-\int_0^x \psi(y)dy=\frac{1}{\Gamma(\alpha/2)}\int_0^t(t-s)^{\alpha/2-1}{\bf E} [\psi(x+Y_{\alpha/2}(s))]ds
%$$
%and
%$$
%{\bf E} \int_{x-Y_{\alpha/2}(t)}^0 \psi(y)dy+\int_0^x \psi(y)dy=\frac{1}{\Gamma(\alpha/2)}\int_0^t(t-s)^{\alpha/2-1}{\bf E} [\psi(x-Y_{\alpha/2}(s))]ds
%$$
%by Example \ref{exm2} and Theorem \ref{thmab},
The solution to the integro-differential equation
\begin{equation}
v(t,x)= \phi(x) + t\psi(x) + \frac{1}{\Gamma(\alpha)}\int_0^t(t-s)^{\alpha-1}v_{xx}(s,x)ds.
\end{equation}
%and differential equation
%\begin{equation}
%\begin{cases}
%D_t^{\alpha}  v(t,x) = \partial_x^2v(t,x), \quad t>0 \\
%v(0,x)= \phi(x), \quad v_t'(0,x)= \psi(x)
%\end{cases}
%\end{equation}
%are equivalent formally. Their solution
could be represented in two different ways:
\begin{eqnarray*}
v(t,x)
= \frac{1}{2}{\bf E}[\phi(x+ Y_{\alpha/2}(t)) + \phi(x-Y_{\alpha/2}(t))]
+\frac{1}{2}g_{2-\alpha/2}(t){\bf E} \int_{x-H_{1}(t)}^{x+H_1(t)} \psi(y)dy
\end{eqnarray*}
and
\begin{eqnarray*}
v(t,x)
&=& \frac{1}{2}{\bf E}[\phi(x+ Y_{\alpha/2}(t)) + \phi(x-Y_{\alpha/2}(t))] \\
    &&+ \frac{1}{2}\int_0^t{\bf E} [\psi(x+Y_{\alpha/2}(s))+\psi(x-Y_{\alpha/2}(s))]ds.
\end{eqnarray*}

If a more compact form for d'Alembert's solution is needed, we should choose the first representation in the case $\beta=\alpha/2$,
 and the second kind representation in the case $\beta=1$.
%in which $H_{1}(t)$ has pdf $(g_{2-\frac{\alpha}{2}}(t))^{-1}(g_{1-\frac{\alpha}{2}}\ast f)(t,y)$.
%And we also have
%$$
%{\bf E} \int_0^{x+H_{1}(t)} \psi(y)dy- \int_0^\infty \int_{0}^{x} \psi(z)dz (g_{1-\frac{\alpha}{2}}\ast f)(t,y)dy=\int_0^t{\bf E} [\psi(x+Y_{\alpha/2}(s))]ds
%$$
%and
%$$
%{\bf E} \int_{x-H_{1}(t)}^0 \psi(y)dy+ \int_0^\infty \int_{0}^{x} \psi(z)dz (g_{1-\frac{\alpha}{2}}\ast f)(t,y)dy=\int_0^t{\bf E} [\psi(x-Y_{\alpha/2}(s))]ds
%$$

%{\bf I think it is wise to compute the numerical solution for  the fundamental solutions with their positive and negative
%components, then picture them as comparation, but I'm unable to do that so far. I'll work on it. If I
%can't work out until submission, I will delete this example. }
%{\bf (to be continued, CGL) It is necessary to give some intricate interpretation for their connection and difference from the point of physics and probability.}

\end{exm}

We will study the d'Alembert solution to differential-difference equations in the next example.

\begin{exm}
Consider the first-order backward difference operator
\begin{equation}
A\phi(x)=\phi(x)-\phi(x-1), \quad\phi\in C_0(\mathbb{R}),\,x\in \mathbb{R}.
\end{equation}
It is clear that $A$ is a bounded operator with $\|A\|=2$, and $A$ generates a
bounded $C_0$-semigroup $e^{tA}$ which is represented as $e^{tA}\phi(x)=e^t\sum_{k=0}^{\infty} \frac{(-t)^k\phi(x-k)}{k!}$.
Therefore, $A$ generates an $\alpha/2$-times resolvent family $S_{\alpha/2}^+(t)$ by the subordination principle:
\begin{equation}
S_{\alpha/2}^+(t) \phi(x)=\int_0^\infty \varphi_{\alpha/2}(t,s)e^{sA}\phi(x)ds, \quad  0<\alpha\leq2
\end{equation}
where $\varphi_{\alpha/2}(t,\cdot)$ is given by (\ref{stablepdf}).
Analogously, the operator $(-A)$ generates a
bounded $C_0$-semigroup $e^{-tA}\phi(x)=e^{-t}\sum_{k=0}^{\infty} \frac{t^k\phi(x-k)}{k!}$ and
$\alpha/2$-times resolvent family
\begin{equation}
S_{\alpha/2}^-(t)\phi(x)=\int_0^\infty \varphi_{\alpha/2}(t,s)e^{-sA}\phi(x)ds.
\end{equation}
In other words, $S_{\alpha/2}^+(t)\phi(x)$ and $S_{\alpha/2}^-(t)\phi(x)$ satisfy the equations
\begin{equation}
\begin{cases}
D_t^{\alpha/2}u^+(t,x) = A u^+(t,x), \quad t>0, \quad x\in\Real \\
u^+(0,x)= \phi(x)
\end{cases}
\end{equation}
and
\begin{equation}
\begin{cases}
D_t^{\alpha/2}u^-(t,x) = -A u^-(t,x), \quad t>0, \quad x\in\Real \\
u^-(0,x)= \phi(x)
\end{cases}
\end{equation}
respectively.

It is easy to see $A^2 \phi(x)=\phi(x)-2\phi(x-1)+\phi(x-2)$. Next we introduce the following wave-type differential-difference equation
\begin{equation}\label{wave-type differential-difference equation}
\begin{cases}
D_t^{\alpha}u(t,x) = A^2 u(t,x), \quad t>0, \quad x\in\Real \\
u(0,x)= \phi(x), \quad u_x(0,x)=0
\end{cases}
\end{equation}
in which $1<\alpha\leq2$. By  Theorem \ref{thmab} the solution of equation (\ref{wave-type differential-difference equation}) can be represented as
\begin{equation}
u(t,x)=\frac12[S_{\alpha/2}^+(t)\phi(x)+S_{\alpha/2}^-(t)\phi(x)].
\end{equation}

\end{exm}

Next we consider the discrete version of Example 3.4, which is related to the time-fractional
Poisson process \cite{Laskin2003, meerschaert2011, Orsinger2013}.
\begin{exm}

For $1<\alpha\leq2$, consider a wave-type differential-difference equation on the integers:
\begin{equation}\label{wave-type differential-difference equation on integer}
\begin{cases}
D_t^{\alpha/2}D_t^{\alpha/2}p(t,k) = p(t,k)-2p(t,k-1)+p(t,k-2), \quad t>0, \quad k\in \mathbb{Z} \\
p(0,0)=1\\
p(0,k)=0, \quad k\neq 0\\
D_t^{\alpha/2}p(0,k)=0.
\end{cases}
\end{equation}

%Define a sequence space
%\begin{equation} \label{solution for wave-type FPP}
%c_0:=\Big\{a(\cdot)\Big|a(k)\geq 0,\sum_{k=-\infty}^{\infty}a(k)=1, k\in \mathbb{Z}\Big\}.
%\end{equation}

%{\bf I don't understand here what space we choose. The space $c_0$ is not a vector space!}

The solution for (\ref{wave-type differential-difference equation on integer}) can be
represented as
\begin{equation}\label{solution for wave-type FPP}
p_\alpha(t,k)=\frac12(p^+_{\alpha/2}(t,k)+p^-_{\alpha/2}(t,k)),
\end{equation}
in which $p^+_{\alpha/2}(t,k)$ and $p^-_{\alpha/2}(t,k)$ are the solutions for
\begin{equation}\label{differential-difference equation for negative FPP }
\begin{cases}
D_t^{\alpha/2}p^+_{\alpha/2}(t,k) = p^+_{\alpha/2}(t,k)-p^+_{\alpha/2}(t,k-1), \quad t>0, \quad k\in \mathbb{Z} \\
p^+_{\alpha/2}(0,0)=1\\
p^+_{\alpha/2}(0,k)=0, \quad k\neq 0
\end{cases}
\end{equation}
and
\begin{equation}\label{differential-difference equation for positive FPP }
\begin{cases}
D_t^{\alpha/2}p^-_{\alpha/2}(t,k) = -\big[p^-_{\alpha/2}(t,k)-p^-_{\alpha/2}(t,k-1)\big], \quad t>0, \quad k\in \mathbb{Z} \\
p^-_{\alpha/2}(0,0)=1\\
p^-_{\alpha/2}(0,k)=0, \quad k\neq 0
\end{cases}
\end{equation}
respectively.

If $k$ is restricted to  non-negative integers, then equation (\ref{differential-difference equation for positive FPP }) governs the time fractional Poisson process,
see \cite{Laskin2003, meerschaert2011, Orsinger2013} and references therein.
The time fractional Poisson process governed by (\ref{differential-difference equation for positive FPP })
is a non-decreasing counting process $N^-(t)$, and $p^-_{\alpha/2}(t,k)$ is the probability
$p^-_{\alpha/2}(t,k)=P(N^-(t)=k)$. Similarly, we can regard equation (\ref{differential-difference equation for negative FPP })
as the governing equation for the (negative) counting process $N^+(t)$ which is a non-increasing process taking values
in the non-positive integers.

Therefore, the d'Alembert solution (\ref{solution for wave-type FPP})
indicates that the stochastic process $N(t)$ governed by (\ref{wave-type differential-difference equation on integer}) is a sum of the
(positive and negative) time fractional Poisson processes $N^+(t)$ and $N^-(t)$, and we may consider $N(t)$ as a wave-type fractional Poisson process.

\end{exm}

The d'Alembert formula solution to evolutionary differential equations can also be
realized in higher dimensions, as we will illustrate in the next example.
The decomposition is related to the fractional Schr\"odinger equation.

\begin{exm}
Consider the following time fractional diffusion-wave equation
\begin{equation}\label{fractional wave-equation in n dimension}
\begin{cases}
D_t^\alpha u(t,x) = \Delta u(t,x), \quad t>0, \,x \in \Real^n,\\
u(0,x)= \phi(x), \quad u_t(0,x)= 0,\\
\lim\limits_{|x|\rightarrow\infty}u(t,x)=0
\end{cases}
\end{equation}
where $1<\alpha\leq 2$ and $\phi\in L^2({\mathbb R}^n)$. For simplicity, we will only study the mild solution on $L^2({\mathbb R}^n)$.

Since $\Delta$ generates cosine function $C(t)$ on $L^2({\mathbb R}^n)$ (\cite{ABHN}),
it also generates an $\alpha$-times fractional resolvent family $S_\alpha (t)$ on $L^2({\mathbb R}^n)$ by subordination principle and
$$
S_\alpha(t) = \int_0^\infty \varphi_{\alpha/2}(t,s) C(s) ds,\quad\quad 1<\alpha\leq2
$$
where $\varphi_{\alpha/2}(t,\cdot)$ is given by (\ref{stablepdf}).
Then the mild solution to (\ref{fractional wave-equation in n dimension}) can be represented as
\begin{equation*}
u(t,x)=S_\alpha(t)\phi(x).
\end{equation*}
Especially, $u(t,x)=C(t)\phi(x)$ when $\alpha=2$.

Define $-(-\Delta)^{\alpha/2}$ as the usual Riesz fractional Laplacian operator, i.e.,
\begin{equation*}
-(-\Delta)^{\alpha/2}f(x)=(2\pi)^{-n}\int_{{\mathbb R}^n}e^{i\xi\cdot x}|\xi|^\alpha\hat f(\xi)d\xi,
\end{equation*}
where $\hat f(\xi)=\int_{{\mathbb R}^n}e^{-i\xi\cdot x}f(x)dx$.

Let $A=i[-(-\Delta)^{1/2}]$ and $-A=-i[-(-\Delta)^{1/2}]$.
Based on the Theorem 3.16.7 in \cite{ABHN}, we obtain the following conclusion:
$(\pm A)^2=\Delta$,  $A$ and $-A$ generate $C_0$-group $U(t)$ and $U(-t)$ on $L^2({\mathbb R}^n)$ respectively, and the group reduction formula
$$
C(t)=\frac12(U(t)+U(-t)).
$$
Therefore the mild solution to (\ref{fractional wave-equation in n dimension}) can also be represented as
\begin{equation}\label{fractional dalembert formula in n dimension}
u(t,x)=\frac12\big(\int_0^\infty \varphi_{\alpha/2}(t,s)U(s)\phi(x)ds+\int_0^\infty \varphi_{\alpha/2}(t,s)U(-s)\phi(x)ds\big)
\end{equation}
which can be regarded as the d'Alembert formula solution in higher dimensions.
Here $\int_0^\infty \varphi_{\alpha/2}(t,s)U(s)\phi(x)ds$ and $\int_0^\infty \varphi_{\alpha/2}(t,s)U(-s)\phi(x)ds$ are solutions to
\begin{equation}\label{fractional u+ in n dimension}
\begin{cases}
D_t^{\alpha/2}u^+(t,x) = A u^+(t,x)=i[-(-\Delta)^{1/2}]u^+(t,x), \quad t>0, \,x \in \Real^n\\
u^+(0,x)= \phi(x)
\end{cases}
\end{equation}
and
\begin{equation}\label{fractional u- in n dimension}
\begin{cases}
D_t^{\alpha/2}u^-(t,x) = -A u^-(t,x)=-i[-(-\Delta)^{1/2}]u^-(t,x), \quad t>0, \,x \in \Real^n\\
u^-(0,x)= \phi(x)
\end{cases}
\end{equation}
respectively.
%It is known that time reversibility for Schr\"odinger equation and spatial-fractional Schr\"odinger equation
%does make sense, so we can regard (\ref{u+ in n dimension}) as another free spatial-fractional Schr\"odinger equation.

In the following we shall consider the special case $\alpha=2$.
When $\alpha=2$, equation (\ref{fractional u- in n dimension}) is a special case of the free space-fractional Schr\"odinger equation
\cite{Bayin,Jeng,Laskin2002}, and (\ref{fractional u+ in n dimension}) is the time reverse of (\ref{fractional u- in n dimension}).
By taking Fourier transforms, we get that the fundamental solutions to (\ref{fractional u+ in n dimension}) and (\ref{fractional u- in n dimension}) can be represented as
\begin{equation*}
u_F^+(t,x)=(2\pi)^{-n}\int_{{\mathbb R}^n}e^{i\xi\cdot x}e^{-i|\xi|t}d\xi
\end{equation*}
and
\begin{equation*}
u_F^-(t,x)=(2\pi)^{-n}\int_{{\mathbb R}^n}e^{i\xi\cdot x}e^{i|\xi|t}d\xi
\end{equation*}
respectively.
%{\bf (\cite{GuoXu} indicates that  $u_F^-(t,x)=\frac{1}{|x|}H^{1,1}_{2,2}\big[(it)^{-1}|x| \mid^{(1,1),(1,1/2)}_{(1,1),(1,1/2)} \big]$, But I don't
%like it.)}

We have not been able to  obtain an explicit closed expression for $u_F^+(t,x)$ or $u_F^-(t,x)$ in the general $n$-dimensional case even in the sense of
distributions. However, we can compute their sum in  one dimension and reduce it to the classical d'Alembert formula for the wave equation.
More precisely, when $n=1$, we have
\begin{eqnarray*}
&&u_F^+(t,x)+u_F^-(t,x)\\
&=&\frac{1}{2\pi}\int_{-\infty}^{+\infty}e^{i\xi x}e^{-i|\xi|t}d\xi+\frac{1}{2\pi}\int_{-\infty}^{+\infty}e^{i\xi x}e^{i|\xi|t}d\xi\\
&=&\frac{1}{2\pi}\int_0^{+\infty}e^{i\xi x}e^{-i\xi t}d\xi+\frac{1}{2\pi}\int_{-\infty}^0e^{i\xi x}e^{i\xi t}d\xi \\
&&+\frac{1}{2\pi}\int_0^{+\infty}e^{i\xi x}e^{i\xi t}d\xi+\frac{1}{2\pi}\int_{-\infty}^0e^{i\xi x}e^{-i\xi t}d\xi \\
&=&\frac{1}{2\pi}\int_{-\infty}^{+\infty}e^{i\xi x}e^{-i\xi t}d\xi+\frac{1}{2\pi}\int_{-\infty}^{+\infty}e^{i\xi x}e^{i\xi t}d\xi\\
&=&\delta(x-t)+\delta(x+t).
\end{eqnarray*}
Therefore the solution (\ref{fractional dalembert formula in n dimension}) to (\ref{fractional wave-equation in n dimension}) is
\begin{equation*}
u(t,x)=\frac12(U(t)\phi(x)+U(-t)\phi(x))=\frac12(\phi(x-t)+\phi(x+t)),
\end{equation*}
which is the same as the classical d'Alembert solution.
\end{exm}

\section{D'Alembert formula for fractional telegraph equations}

It was obtained by Kac \cite{Kac,Kac1974} that for the telegraph equation
\begin{equation}\label{telegraph equation}
\begin{cases}
   \frac{\partial^2 u(t,x)}{\partial t^2} +2h\frac{\partial u(t,x)}{\partial t}= c^2 \frac{\partial ^2 u(t, x)}{\partial x^2}, \quad \quad \quad t>0; \\
   u(0, x)=\phi(x),\quad \quad u_t'(0,x) =0.
\end{cases}
\end{equation}
the solution is given by
\begin{equation}\label{solution to telegraph equation}
u(t,x) = \frac{1}{2} \mathbf{E}[\phi(x + c\xi_h(t) )+ \phi(x - c\xi_h(t) )],
\end{equation}
where
$$
 \xi_h(t) = \int_0^t (-1)^{N_h(s)}ds , \quad t \ge 0
$$
and $N_h(t)$ is a Poisson process with parameter $h$, that is, $\mathbb{P}\{N_h(t) = k\} = \frac{h^k t^k}{k!}e^{-ht}$.
When $h=0$, we have $N_h(t)\equiv 0$ and $\xi_h(t)=t$, then (\ref{telegraph equation}) and (\ref{solution to telegraph equation})
are reduced to the classical wave equation and d'Alembert solution respectively. Therefore, (\ref{solution to telegraph equation})
can be regarded as the d'Alembert formula solution for the telegraph equation.

We may also give the d'Alembert's formula solution for the (sequential Caputo) fractional telegraph equation  with proper initial conditions.
Let $1<\alpha\leq 2$, consider the following fractional telegraph equation:
\begin{equation}\label{fractional telegraph equation}
\begin{cases}
   D_t^{\alpha/2}D_t^{\alpha/2}u(t)+2hD_t^{\alpha/2}u(t)=Au(t) \quad \quad \quad t>0; \\
   u(0)=\phi,\quad \quad D_t^{\alpha/2} u (0)=\psi.
\end{cases}
\end{equation}

Formally, the equation (\ref{fractional telegraph equation}) can be rewritten as
\begin{eqnarray*}
\begin{cases}
(D^{\alpha/2}_t+h)^2u(t)=(A+h^2)u(t)=B^2 u(t)\\
u(0)=\phi,\quad \quad D_t^{\alpha/2} u (0)=\psi,
\end{cases}
\end{eqnarray*}
if there is an operator $B$ satisfying $A+h^2 = B^2$. Therefore, the operator   $D_t^{\alpha/2}D_t^{\alpha/2} + 2h D_t^{\alpha/2}-A$ can be decomposed into the product of
$D^{\alpha/2}_t+h+B$ and  $D^{\alpha/2}_t+h - B$. We will then study the relations between the solution of (\ref{fractional telegraph equation}) and the solutions of
\begin{equation}\label{fractional telegraph +}
\begin{cases}
D^{\alpha/2}_tu^+_{\alpha/2}(t)=(B-h)u^+_{\alpha/2}(t),\\
u^+_{\alpha/2}(0)=\phi+B^{-1}(\psi+h\phi)
\end{cases}
\end{equation}
and
\begin{equation}\label{fractional telegraph -}
\begin{cases}
D^{\alpha/2}_tu^-_{\alpha/2}(t)=(-B-h)u^-_{\alpha/2}(t),\\
u^-_{\alpha/2}(0)=\phi-B^{-1}(\psi+h\phi)
\end{cases}
\end{equation}
when $\psi+h\phi$ is in the range of $B$.

%Next we will try to give the solutions to equations (\ref{fractional telegraph +}) and (\ref{fractional telegraph -}). This is possible if
%both $B-h$ and $B+h$ generate $\alpha/2$-times resolvent families.
%In the next result the (non-standard) d'Alembert solution for equation (\ref{fractional telegraph equation}) is constructed via fractional resolvent families.

\begin{thm}\label{thmFTE}
Suppose that $B$ is an operator satisfying $B^2 = A+ h^2$, and both $B-h$ and $-B-h$ generate $\alpha/2$ times resolvent families
$S^+_{\alpha/2}(t)$ and $S^-_{\alpha/2}(t)$ on $X$, respectively.
Then for $\phi \in D(A)$, and $\psi+ h\phi = B f $ for some $f \in D(A^2)$ the strong solutions of (\ref{fractional telegraph +}) and (\ref{fractional telegraph -})
can be expressed as
\begin{equation}
u^+_{\alpha/2}(t)=S^+_{\alpha/2}(t)\phi+ S^+_{\alpha/2}(t)B^{-1}(\psi+h\phi)
\end{equation}
and
\begin{equation}
u^-_{\alpha/2}(t)=S^-_{\alpha/2}(t)\phi- S^-_{\alpha/2}(t)B^{-1}(\psi+h\phi)
\end{equation}
respectively, and the d'Alembert's formula solution for (\ref{fractional telegraph equation}) reads:
\begin{equation}\label{erheu}
u_{\alpha}(t)=\frac{1}{2}[S^+_{\alpha/2}(t)\phi+S^-_{\alpha/2}(t)\phi]
              +\frac{1}{2}[S^+_{\alpha/2}(t)B^{-1}(\psi+h\phi)- S^-_{\alpha/2}(t)B^{-1}(\psi+h\phi)].
\end{equation}
\end{thm}
\begin{proof}
First note that both $\phi$ and $B^{-1}(\psi + h\phi)$ are in the domain of $B^2$. Let
$
u_{\alpha}(t)=\frac12(u^+_{\alpha/2}(t)+u^-_{\alpha/2}(t)).
$
Then we have
\begin{eqnarray*}
&&(D^{\alpha/2}_t+h)^2u_{\alpha}(t)\\
&=&\frac12(D^{\alpha/2}_t+h)(D^{\alpha/2}_t+h)(u^+_{\alpha/2}(t)+u^-_{\alpha/2}(t)) \\
&=&\frac12(D^{\alpha/2}_t+h)(D^{\alpha/2}_t+h)u^+_{\alpha/2}(t)
   +\frac12(D^{\alpha/2}_t+h)(D^{\alpha/2}_t+h)u^-_{\alpha/2}(t) \\
&=&\frac12(D^{\alpha/2}_t+h)Bu^+_{\alpha/2}(t)
   -\frac12(D^{\alpha/2}_t+h)Bu^-_{\alpha/2}(t) \\
&=&\frac12B(D^{\alpha/2}_t+h)u^+_{\alpha/2}(t)
   -\frac12B(D^{\alpha/2}_t+h)u^-_{\alpha/2}(t) \\
&=&\frac12B^2(u^+_{\alpha/2}(t)+u^-_{\alpha/2}(t)) \\
&=&B^2u_{\alpha}(t).
\end{eqnarray*}
Moreover,
\begin{eqnarray*}
u_{\alpha}(0)=\frac12[u^+_{\alpha/2}(0)+u^-_{\alpha/2}(0)]
             =\frac12[\phi+B^{-1}(\psi+h\phi)+\phi-B^{-1}(\psi+h\phi)]
             =\phi.
\end{eqnarray*}
Finally, we have
\begin{eqnarray*}
2\lim\limits_{t\rightarrow0}(D_t^{\alpha/2}u_{\alpha})(t)
&=&\lim\limits_{t\rightarrow0}(D_t^{\alpha/2}u^+_{\alpha/2})(t)+\lim\limits_{t\rightarrow0}(D_t^{\alpha/2}u^-_{\alpha/2})(t)\\
&=&Bu^+_{\alpha/2}(0)-hu^+_{\alpha/2}(0)-Bu^-_{\alpha/2}(0)-hu^-_{\alpha/2}(0)\\
&=&B[B^{-1}(\psi+h\phi)+B^{-1}(\psi+h\phi)]-h(u^+_{\alpha/2}(0)+u^-_{\alpha/2}(0))\\
&=&2\psi+2h\phi-2h\phi
=2\psi.
\end{eqnarray*}
Thus (\ref{erheu}) is the strong solution of (\ref{fractional telegraph equation}).
\end{proof}

\begin{rem}\label{rem4.2}
(1) When $h=0$, Theorem \ref{thmFTE} reduces to Theorem \ref{abFujita}.

(2) Suppose that $-(\omega_0 +A)$ is a sectorial operator with angle $0<\varphi < \pi - \pi \alpha $
for some constant $\omega_0 >0$. Then $h^2 +A$
generates a bounded analytic $\alpha$-times resolvent family
for each constant $h$ satisfying $|h| \le \sqrt{\omega_0}$.
Let $B =i[-(h^2+A)]^{1/2}$, then $(\pm B)^2=A+h^2$, $B$ is invertible, and $\pm B$ generate bounded analytic $\alpha/2$-times resolvent families.
Therefore, both $B-h$ and $-B-h$ generate exponentially bounded $\alpha/2$-times resolvent family.
See \cite[Section 7]{lichenggang2012} for more details.
\end{rem}

\begin{exm}
Let $1<\alpha\leq 2$. Consider
\begin{equation}\label{new concrete fractional telegraph equation}
\begin{cases}
   D_t^{\alpha/2}D_t^{\alpha/2}u(t,x)+2hD_t^{\alpha/2}u(t,x)=c^2\frac{\partial^2u(t,x)}{\partial x^2} \quad \quad \quad t>0,\quad x\in\mathbb{R}; \\
   u(0,x)=\phi(x),\quad \quad D_t^{\alpha/2} u (0,x)=0.
\end{cases}
\end{equation}
which is equivalent to
\begin{equation}\label{old concrete fractional telegraph equation}
\begin{cases}
   D_t^{\alpha}u(t,x)+2hD_t^{\alpha/2}u(t,x)=c^2\frac{\partial^2u(t,x)}{\partial x^2} \quad \quad \quad t>0,\quad x\in\mathbb{R}; \\
   u(0,x)=\phi(x),\quad \quad  u_t(0,x)=0.
\end{cases}
\end{equation}

Denote by $A=c^2\frac{d^2}{dx^2}$. Then the conditions in Theorem \ref{thmFTE} and Remark \ref{rem4.2} are fulfilled.
Therefore the solution to (\ref{new concrete fractional telegraph equation}) can be represented as
\begin{equation*}
u(t,x)=\frac{1}{2}[S^+_{\alpha/2}(t)(1+hB^{-1}) \phi(x)+S^-_{\alpha/2}(t)(1-hB^{-1}) \phi(x)]
\end{equation*}
%By taking Laplace transform, we have
%\begin{equation*}
%\hat u(\lambda,x)=(\lambda^{\alpha-1}+2h\lambda^{\alpha/2-1})(\lambda^{\alpha}+2h\lambda^{\alpha/2}-A)^{-1}\phi(x).
%\end{equation*}
%Moreover, taking Fourier transform leads to
%\begin{equation*}
%\tilde{\hat u}(\lambda, \xi)=\int_{-\infty}^{\infty} e^{i\xi x}\hat u(\lambda,x)dx=\frac{\lambda^{\alpha-1}+2h\lambda^{\alpha/2-1}}{\lambda^{\alpha}+2h\lambda^{\alpha/2}+c^2\xi^2}\phi(\xi)
%\end{equation*}
%which is the formula (2.6) in \cite{Orsinger2004}. By putting inverse Laplace transform
%on $\tilde{\hat u}(\lambda, \xi)$,
Orsingher and Beghin (\cite{Orsinger2004}) have obtained the following Fourier transform
of the solution to (\ref{old concrete fractional telegraph equation}):
\begin{eqnarray*}
\tilde{u}(t, \xi)&=&E_{\alpha,1}(\eta_1 t^\alpha)+\frac{(2h+\eta_2)t^\alpha}{\eta_1-\eta_2}
                   [\eta_1E_{\alpha,\alpha+1}(\eta_1 t^\alpha)-\eta_2E_{\alpha,\alpha+1}(\eta_2 t^\alpha)]\phi(\xi)\\
  &=&\frac12 \big[ \big(1+\frac{h}{\sqrt{h^2-c^2\xi^2}}\big)E_{\alpha,1}(\eta_1 t^\alpha)
     + \big(1-\frac{h}{\sqrt{h^2-c^2\xi^2}}\big)E_{\alpha,1}(\eta_2 t^\alpha)\big]\phi(\xi)
\end{eqnarray*}
where
$$
\eta_1=-h+\sqrt{h^2-c^2\xi^2},\quad\quad \eta_2=-h-\sqrt{h^2-c^2\xi^2}.
$$
They also remarked that, for $\alpha=2$, $\tilde{u}(t, \xi)$
can be reduced to the characteristic function of the telegraph process
$$
T(t)=V(0)\xi_h(t) =V(0) \int_0^t (-1)^{N_h(s)}ds,
$$
where $V(0)$ is a random variable  with the probability distribution
$P(V(0)=\pm c)=1/2$ independent of $N_h(t)$.

On the other hand, we can obtain the d'Alembert formula solution for (\ref{new concrete fractional telegraph equation})
by Bernstein and subordinator theory \cite{Applebaum2009, Schilling2012}. Indeed, since  $z(\lambda)=(\lambda^{\alpha}+2h\lambda^{\alpha/2})^{1/2}$
is a Bernstein function for $0<\alpha\leq2,\, h\geq0$, there is a subordinator $D_z(t)$ with Laplace exponent $z(\lambda)$, that is to say,
\begin{equation*}
\mathbb{E}\big(e^{-\lambda D_z(t)}\big)=e^{-t z(\lambda)}.
\end{equation*}
Denote the inverse of $D_z(t)$ by $Y_z(t)$,i.e.,$Y_z(t)=inf\{y>0:D_z(y)>t\}$.  Then we have
\begin{equation} \label{Laplace exponent}
 \int_0^\infty e^{-\lambda t}dF(t,x)=\frac{z(\lambda)}{\lambda}e^{-xz(\lambda)}, \quad x\geq0
\end{equation}
where $F(t,x)$ is the distribution function of $Y_z(t)$.

Thanks to the subordination principle and the d'Alembert's solution (\ref{dAlmbert}) to (\ref{wave-equation}),
the solution to (\ref{new concrete fractional telegraph equation}) can be represented as
\begin{equation} \label{LCG-subordination solution to FTE}
u(t,x)=\frac12\mathbb{E}[\phi(x+cY_z(t))+\phi(x-cY_z(t))]
\end{equation}

Next, we will clarify the the relation between (\ref{solution to telegraph equation}) and (\ref{LCG-subordination solution to FTE}) in the special case $\alpha=2$.
Let $z(\lambda)=\sqrt{\lambda^2+2h\lambda}$.
Denote the density of $\xi_h(t)$ by $g(t,x)$, $x\in(-\infty,\infty)$. Then $|\xi_h(t)|$ has density
$$
w(t,r)=\begin{cases} g(t,x)+g(t,-x), \quad x\geq 0,\\
0,\quad\quad\quad\quad\quad\quad\quad x<0.
\end{cases}
$$
Moreover, we have \cite{Dewitt1989}
\begin{eqnarray*}
 \int_0^\infty e^{-\lambda t}g(t,x)dt=
\begin{cases}
 \frac12[\frac{z(\lambda)}{\lambda}+1] e^{-xz(\lambda)}, \quad \quad   x>0,\\
 \frac12[\frac{z(\lambda)}{\lambda}-1] e^{xz(\lambda)}, \quad \quad  \quad x<0,\\
 \frac{z(\lambda)}{2\lambda}, \quad \quad \quad \quad \quad \quad \quad  \quad x=0,
\end{cases}
\end{eqnarray*}
and
\begin{equation}\label{Laplace exponent 2}
 \int_0^\infty e^{-\lambda t}w(t,x)dt=\frac{z(\lambda)}{\lambda}e^{-xz(\lambda)},\quad\quad x\geq 0.
\end{equation}
Since the solution of (\ref{telegraph equation}) can be represented as (\ref{solution to telegraph equation}):
\begin{eqnarray*}
u(t,x)&=&\frac12 \mathbb{E}[\phi(x+c\xi_h(t))+\phi(x-c\xi_h(t))] \\
      &=&\int_{-\infty}^{\infty}\frac12 [\phi(x+r)+\phi(x-r)]g(t,r)dr ,
\end{eqnarray*}
the solution of (\ref{telegraph equation}) can be represented in another way:
\begin{eqnarray*}
u(t,x)&=&\frac12 \mathbb{E}[\phi(x+c|\xi_h(t)|)+\phi(x-c|\xi_h(t)|)] \\
      &=&\int_{0}^{\infty}\frac12 [\phi(x+r)+\phi(x-r)]w(t,r)dr .
\end{eqnarray*}

It is easy to see $|\xi_h(t)|$ and the inverse subordinator $Y_z(t)$ are identically distributed from (\ref{Laplace exponent})
and (\ref{Laplace exponent 2}).
\end{exm}
%\begin{rem}
%Formally, the operator $B-h=\sqrt{A^2+h^2}-h$. Taking  $A^2$ as $-\Delta$, let $h=mc^2$, then $B-h$
%is the relativistic Schr\"odinger operator \cite{carmona1990}, where $m$ is the rest mass and $c$ is the speed of light.
%\end{rem}


\begin{thebibliography}{\bf\large99}
\bibitem{Applebaum2009} D. Applebaum, L{\'e}vy Processes and Stochastic Calculus, Cambridge University Press, 2009.

\bibitem{ABHN} W. Arendt, C.J.K. Batty, M. Hieber and  F. Neubrander, Vector-Valued Laplace Transforms and Cauchy Problems, Birkh\" {a}user, 2010.

\bibitem{forcing}  B. Baeumer, M.M. Meerschaert and S. Kurita, Inhomogeneous fractional diffusion equations.
{\it Fract. Calc.  Appl. Anal.} {\bf 8}(4) (2005), 371--386.

\bibitem{Baj} E. G. Bajlekova, Fractional evolution equations in Banach spaces, PhD thesis, Department of mathematics, Eindhoven University of Technology, 2001.

\bibitem{Bayin} S.S.Bayin, Definition of the Riesz derivative and its application to space fractional quantum mechanics.
 {\it J. Math. Phys.} {\bf 57}(12) (2016), 123501.

\bibitem{CLL} C. Chen, M. Li and  F.B. Li, On boundary values of fractional resolvent families, {\it J. Math. Anal. Appl.} {\bf 384} (2011), 453-467.

\bibitem{Dewitt1989} C. DeWitt-Morette, S.K. Foong, Path-integral solutions of wave equations with dissipation.
{\it Phys. Rev. Lett. } {\bf 62} (19) (1989), 2201-2204.

%7\bibitem{carmona1990} R.Carmona, Path integrals for relativistic Schr\"odinger operators. Schr\"odinger operators. Springer, Berlin, Heidelberg, 1989. 65-92.

\bibitem{Engel1999} K.J. Engel and R. Nagel, One-Parameter Semigroups for Linear Evolution Equations, GTM {\bf 194}, Springer Science and Business Media, 1999.
%\bibitem{godinho2012} C.F.L.Godinho,J. Weberszpil,J.A.Helay\"el-Neto, Extending the D'alembert solution to space-time Modified Riemann-Liouville fractional wave equations, {\it Chaos, Solitons, Fractals.} {\bf
%45(6)},(2012), 765-771.

\bibitem{Fattorini} H.O. Fattorini, Ordinary differential equations in linear topological spaces, II, {\it J. Differential Equations} {\bf 6} (1969), 50-70.

\bibitem{Fujita} Y. Fujita, Integrodifferential equation which interpolates the heat euqation and the wave equation (II), {\it Osaka J. Math.} {\bf 27} (1990), 797-804.

\bibitem{gorenflo2015} R. Gorenflo, Stochastic processes related to time-fractional diffusion-wave equation,
{\it Commun. Appl. Ind. Math.} {\bf  6} (2) (2015), e-531, 8 pp.

%\bibitem{GuoXu} X.Y. Guo and M.Y. Xu, Some physical applications of fractional Schr\"odinger equation.
%{\it J. Math. Phys.} {\bf 47} (2006), 082104.

%\bibitem{HR} R. Hilfer,  Applications of Fractional Calculus in Physics, World Scientific, Singapore, 2000.

\bibitem{Jeng} M. Jeng, S.L.Y. Xu, E. Hawkins and J.M. Schwarz, On the nonlocality of the fractional Schr\"odinger equation,
{\it J. Math. Phys.} {\bf 51}(6) (2010), 062102.

\bibitem{Kac} M. Kac, Some Stochastic Problems in Physics and Mechanics, Magnolia Petrolum Co. Colloq. Lect. 2, 1956.

\bibitem{Kac1974} M. Kac, A stochastic model related to the telegrapher's equation,
{\it Rocky Mountain J. Math.} {\bf 4}(3) (1974), 497-510.

\bibitem{Kilbas} A.A. Kilbas, H.M. Srivastava, and J.J. Trujillo, Theory and Applications of Fractional Differential Equations,
Elsevier, New York, 2002.

\bibitem{Krein} S.G. Krein, Linear Differential Equations in Banach space, American Mathematical Society, Providence, R.I., 1971. Translated from the Russian by J. M. Danskin, Translations of Mathematical
Monographs, Vol. 29.


\bibitem{Laskin2002} N. Laskin, Fractional Schr\"odinger equation. {\it Phys. Rev. E}, {\bf 66}, (2002), 056108.

\bibitem{Laskin2003} N. Laskin, Fractional Poisson process, {\it Commun. Nonlinear Sci. Numer. Simul.}
{\bf 8} (2003), 201-213.

\bibitem{LiPeng} K.X. Li, J.G. Peng and J.X. Jia, Cauchy problems for fractional differential equations with Riemann-Liouville fractional derivatives,
{\it J. Funct. Anal.} {\bf 263} (2012), 476-510.

\bibitem{lichenggang2012} C.G. Li, M. Kostic, M. Li, and S. Piskarev, On a class of time-fractional differential equations,
{\it Fract. Calc. Appl. Anal.} {\bf 15} (4) (2012), 639-668.

\bibitem{LCL} M. Li, C. Chen and F.B. Li, On fractional powers of generators of fractional resolvent families, {\it J. Funct. Anal. } {\bf 259} (2010), 2702-2726.

\bibitem{LZ} M. Li and Q. Zheng, On spectral inclusions and approximations of  $\alpha$-times resolvent families,  {\it Semigroup Forum} {\bf 69}(3) (2004), 356-368.

\bibitem{meerschaert2011} M.M. Meerschaert, E. Nane and P. Vellaisamy,  The fractional Poisson process and the inverse stable subordinator, {\it Electron. J. Probab.} {\bf 16} (2011), 1600-1620.

\bibitem{meerschaert} M.M. Meerschaert, R.L. Schilling and A.Sikorskii, Stochastic solutions for fractional wave equations,
{\it Nonlinear Dyn.} {\bf 80}(4) (2015), 1685-1695.

\bibitem{meerschaert2012book} M.M. Meerschaert, A. Sikorskii, Stochastic Models for Fractional Calculus,
De Gruyter Studies in Mathematics {\bf 43},  Walter de Gruyter Co., Berlin, 2012.

\bibitem{meerschaert2013} M.M. Meerschaert and F.Straka, Inverse stable subordinators,
{\it Math. Model. Nat. Phenom.} {\bf 8}(2) (2013), 1-16.

%\bibitem{Naber} M.Naber. Time fractional Schr\"odinger equation, {\it Journal of Mathematical Physics}, {\bf 45}(8),(2004),3339-3352.

\bibitem{Orsinger2013} E. Orsingher, Fractional Poisson processes, {\it Sci. Math. Jpn} {\bf 76}(1) (2013), 139-145.

\bibitem{Orsinger2004} E. Orsingher and L. Beghin, Time-fractional telegraph equations and telegraph processes with Brownian time.
 {\it Probab. Theory Related Fields } {\bf 128}(1) (2004),141-160.

\bibitem{Pazy} A. Pazy, Semigroups of Linear Operators and Applications to Partial Differential Equations, Springer, New York, 1983.

\bibitem{Podlubny} I. Podlubny, Fractional Differential Equations, Academic Press, New York, 1999.

\bibitem{Pruss} J. Pr\"uss, Evolutionary Integral Equations and Applications, Birkh\"auser, Basel, 1993.

\bibitem{Samko} S.G. Samko, A.A. Klibas and O.I. Marichev, Fractional Integrals and Derivatives, Theory and Applications, Gordon and Breach Science Publishers, 1992.

\bibitem{Schilling2012} R.L. Schilling, R. Song and Z. Vondracek, Bernstein Functions: Theory and Applications, Walter de Gruyter, 2012.

%\bibitem{TV} V. Tarasov, Fractional Dynamics: Applications of fractional Calculus to Dynamics of Particles, Fields and Media, Springer, New York, 2011.


\bibitem{Umarov} S. Umarov, On fractional Duhamel's principle and its
applications, {\it J. Differential Equations} {\bf 252} (10) (2012), 5217-5234.

\end{thebibliography}
\end{document}